\pdfoutput=1
\newif\ifpersonal
\RequirePackage[l2tabu,orthodox]{nag} 
\documentclass[12pt,a4paper,reqno]{amsart} 
\usepackage{amsmath,amsthm,amssymb,mathrsfs,mathtools,bm,eucal,tensor} 
\usepackage{microtype,fixltx2e,lmodern} 
\usepackage[utf8]{inputenc} 
\usepackage[T1]{fontenc} 
\usepackage{enumerate,comment,braket,xspace,tikz-cd,csquotes} 
\usepackage[all,cmtip]{xy} 
\usepackage[centering,vscale=0.7,hscale=0.7]{geometry}
\usepackage[hidelinks]{hyperref}
\usepackage[capitalize]{cleveref}

\theoremstyle{plain}
\newtheorem{thm-intro}{Theorem}
\newtheorem{thm}{Theorem}[section]
\newtheorem*{thm*}{Theorem}
\newtheorem{lem}[thm]{Lemma}
\newtheorem{prop}[thm]{Proposition}

\newtheorem{cor}[thm]{Corollary}

\theoremstyle{definition}
\newtheorem{defin}[thm]{Definition}

\theoremstyle{remark}

\newtheorem{rem}[thm]{Remark}
\numberwithin{equation}{section}

\ifpersonal
\newcommand*{\personal}[1]{\textcolor[rgb]{0.6,0.6,1}{(Personal: #1)}}
\newcommand*{\todo}[1]{\textcolor{red}{(Todo: #1)}}
\else
\newcommand*{\personal}[1]{\ignorespaces}
\newcommand*{\todo}[1]{\ignorespaces}
\fi

\DeclareFontFamily{U}{BOONDOX-calo}{\skewchar\font=45 }
\DeclareFontShape{U}{BOONDOX-calo}{m}{n}{<-> s*[1.05] BOONDOX-r-calo}{}
\DeclareFontShape{U}{BOONDOX-calo}{b}{n}{<-> s*[1.05] BOONDOX-b-calo}{}
\DeclareMathAlphabet{\mathcalboondox}{U}{BOONDOX-calo}{m}{n}

\makeatletter
\let\save@mathaccent\mathaccent
\newcommand*\if@single[3]{%
	\setbox0\hbox{${\mathaccent"0362{#1}}^H$}%
	\setbox2\hbox{${\mathaccent"0362{\kern0pt#1}}^H$}%
	\ifdim\ht0=\ht2 #3\else #2\fi
}
\newcommand*\rel@kern[1]{\kern#1\dimexpr\macc@kerna}
\newcommand*\widebar[1]{\@ifnextchar^{{\wide@bar{#1}{0}}}{\wide@bar{#1}{1}}}
\newcommand*\wide@bar[2]{\if@single{#1}{\wide@bar@{#1}{#2}{1}}{\wide@bar@{#1}{#2}{2}}}
\newcommand*\wide@bar@[3]{%
	\begingroup
	\def\mathaccent##1##2{%
		\let\mathaccent\save@mathaccent
		\if#32 \let\macc@nucleus\first@char \fi
		\setbox\z@\hbox{$\macc@style{\macc@nucleus}_{}$}%
		\setbox\tw@\hbox{$\macc@style{\macc@nucleus}{}_{}$}%
		\dimen@\wd\tw@
		\advance\dimen@-\wd\z@
		\divide\dimen@ 3
		\@tempdima\wd\tw@
		\advance\@tempdima-\scriptspace
		\divide\@tempdima 10
		\advance\dimen@-\@tempdima
		\ifdim\dimen@>\z@ \dimen@0pt\fi
		\rel@kern{0.6}\kern-\dimen@
		\if#31
		\overline{\rel@kern{-0.6}\kern\dimen@\macc@nucleus\rel@kern{0.4}\kern\dimen@}%
		\advance\dimen@0.4\dimexpr\macc@kerna
		\let\final@kern#2%
		\ifdim\dimen@<\z@ \let\final@kern1\fi
		\if\final@kern1 \kern-\dimen@\fi
		\else
		\overline{\rel@kern{-0.6}\kern\dimen@#1}%
		\fi
	}%
	\macc@depth\@ne
	\let\math@bgroup\@empty \let\math@egroup\macc@set@skewchar
	\mathsurround\z@ \frozen@everymath{\mathgroup\macc@group\relax}%
	\macc@set@skewchar\relax
	\let\mathaccentV\macc@nested@a
	\if#31
	\macc@nested@a\relax111{#1}%
	\else
	\def\gobble@till@marker##1\endmarker{}%
	\futurelet\first@char\gobble@till@marker#1\endmarker
	\ifcat\noexpand\first@char A\else
	\def\first@char{}%
	\fi
	\macc@nested@a\relax111{\first@char}%
	\fi
	\endgroup
}
\makeatother

\usetikzlibrary{decorations.markings} 
\tikzset{
  closed/.style = {decoration = {markings, mark = at position 0.5 with { \node[transform shape, xscale = .8, yscale=.4] {/}; } }, postaction = {decorate} },
  open/.style = {decoration = {markings, mark = at position 0.5 with { \node[transform shape, scale = .7] {$\circ$}; } }, postaction = {decorate} }
}

\begin{document}
\title{Cosheaf Theory and Overconvergent Verdier Duality for Rigid Analytic Spaces}

\author{Vaibhav Murali}
\address{}

\date{Nov 19, 2020}

\begin{abstract}
This paper develops aspects of cosheaf theory on rigid analytic spaces, and demonstrates a sheaf-cosheaf Verdier duality equivalence theorem for overconvergent sheaves on separated, paracompact spaces, analogous to Jacob Lurie's treatment of Verdier duality for locally compact, Hausdorff topological spaces.
\end{abstract}

\maketitle

\tableofcontents

\section{Introduction}
\paragraph{\textbf{Broad Overview}}

Poincar\'{e} duality is a classical result concerning the cohomology of manifolds, and is among a large number of duality theorems in the literature for various types of cohomology (such as Serre duality), which are often formulated in terms of the existence of a dualizing sheaf for a space. Heuristically, letting \(M\) be a manifold or similar such space, we would like the dualizing sheaf to be a sheaf \(\omega_M\) on \(M\) so that we have a natural identification roughly of form

\[\mathrm{Maps}(\Gamma_c(\underline{\mathbb{R}}),\mathbb{R}) \cong \mathrm{Maps}(\underline{\mathbb{R}},\omega_M),\]

where \(\underline{\mathbb{R}}\) denotes the constant sheaf, and \(\Gamma_c\) denotes (derived) compactly supported global sections.
Verdier duality is traditionally understood to be a vast generalization of Poincar\'{e} duality from the setting of manifolds to that of arbitrary (maps of) locally compact, Hausdorff topological spaces. The basic idea is to think of the existence of dualizing sheaves in terms of the existence of a right adjoint (often called the \itshape shriek-pullback \upshape or \itshape exceptional inverse image \upshape) to a variant of the traditional direct image of sheaves (when the target is a point, this is just taking global sections), namely the shriek (!)-pushforward or direct image with proper support (generalizing taking compactly supported global sections). In the above picture, \(\omega_M\) is the shriek-pullback of the constant sheaf \(\underline{\mathbb{R}}\) on a point. Among the perspectives on Verdier duality, one specifically related to this paper is that given by demonstrating the existence of an equivalence between appropriate categories of sheaves and their dual counterpart \itshape cosheaves, \upshape given by passage to compactly supported sections. The relevance of this picture is that, just as direct images of sheaves have left adjoints, direct images of cosheaves readily have right adjoints; under a sheaf-cosheaf equivalence theorem, we can think of the cosheaf direct image and its right adjoint as corresponding to the shriek pushforward and pullback of sheaves respectively.

\par
The traditional setting for Verdier duality deals with derived categories of sheaves and formulates the result in terms of derived direct images (with proper support) and existence of an appropriate derived right adjoint. With the modern language of \(\infty\)-categories, Jacob Lurie formulates a much more general result, replacing derived categories of sheaves with sheaves valued in a very general class of stable \(\infty\)-categories (those with small limits and colimits), which, for instance, could include derived \(\infty\)-categories such as that of an abelian category. (Note to the reader new to these issues: there is a subtlety that Lurie works with a generalization of sheaves valued in an appropriate derived \(\infty\)-category, which is not obviously the same as a derived \(\infty\)-category of sheaves: without suitable boundedness assumptions, derived \(\infty\)-categories of sheaves correspond to \itshape hypercomplete \upshape sheaves valued in appropriate derived \(\infty\)-categories.)
Jacob Lurie's result is replicated here as follows:

\begin{thm}
Let \(X\) be a locally compact, Hausdorff topological space, and let \(\mathcal{C}\) be a stable \(\infty\)-category with all small limits and colimits. For \(U \subset X\) open and \(\mathcal{F}\) a sheaf on \(X\), denote by \(\Gamma_c(U,\mathcal{F})\) the sections with compact support of \(\mathcal{F}\) over \(U\). Then, the functor 

\[\Gamma_c(-): \mathrm{Shv}(X,\mathcal{C})\rightarrow \mathrm{CShv}(X,\mathcal{C})\]

given by sending a sheaf \(\mathcal{F}\) to the cosheaf \(\mathcal{G}\) given by \((U \subset X) \mapsto \Gamma_c(U,\mathcal{F})\) is an equivalence of \(\infty\)-categories. 
\end{thm}

There are various discussions of duality theorems in rigid analytic (more generally, nonarchimedean analytic) geometry, analogous to the Poincar\'{e} or Serre duality theorems for manifolds in the literature, and Huber even has developed a certain type of Verdier duality in the setting of adic spaces. What distinguishes the present paper is the focus on finding a version of a sheaf-cosheaf equivalence of categories out of which the appropriate duality theory falls out, an approach which allows us to consider highly general coefficient types (again, sheaves valued in a stable \(\infty\)-category with small limits and colimits). To closely replicate Lurie's result in the rigid analytic setting, the author was led to restrict to the class of overconvergent sheaves specifically.
\par
Overconvergent sheaves arise naturally for consideration in nonarchimedean analytic geometry, due to the fact that the usual sheaves are not determined by their behavior on rigid analytic analogue of open subsets in manifold theory (wide open subsets), and their behavior on the analogue of compact subsets is essential. It is known that overconvergence can aid in the proving of duality theorems, for instance as in Elmar Grosse-Kl\"{o}nne's \itshape Rigid Analytic Spaces with Overconvergent Structure Sheaf, \upshape where Serre and Poincar\'{e} duality theorems are proved for quasi-compact spaces where they would usually be proved for partially proper ones.
\par
With this background, we now state our main result (an analogue of Lurie's above one):

\begin{thm}
Let \(X\) be a separated, paracompact rigid analytic space, and let \(\mathcal{C}\) be a stable \(\infty\)-category with small limits and colimits. Then, there is a functor

\[\Gamma_c(-):\mathrm{OverShv}(X,\mathcal{C}) \rightarrow \mathrm{OverCShv}(X,\mathcal{C})\]

defined by sending an overconvergent sheaf \(\mathcal{F}\) to the cosheaf given on admissible open \(U\) by \((U \mapsto \Gamma_c(U,\mathcal{F}))\), and it is an equivalence of \(\infty\)-categories.

\end{thm}

\paragraph{\textbf{Detailed Discussion of Contents of the Paper}}

We now provide a detailed overview of the contents of the various sections of the paper. 
\par
The goal of the first mathematical section 2 is to provide some foundation for later results by demonstrating the relation (equivalence) between sheaves (analogous remarks apply for cosheaves) on rigid analytic spaces and a more restricted analogue of sheaves called \(\mathcal{K}\)-sheaves. We define the notion of a \(\mathcal{K}\)-sheaf on a rigid analytic space \(X\), which is a functor defined only on quasi-compact admissible opens, and which satisfies descent for 2-element coverings. The significance of \(\mathcal{K}\)-sheaves for us is basically that, since the main functor we are interested in is one given by passing to quasi-compact supports, it will be useful to have a description of sheaves and cosheaves entirely depending on their behavior on quasi-compact admissible opens. Denoting by \(\mathrm{Shv}_{\mathcal{K}}(X,\mathcal{C})\) the \(\infty\)-category of \(\mathcal{K}\)-sheaves, the main result of the section is Proposition 2.7:

\begin{prop}
Restriction induces an equivalence of \(\infty\)-categories

\[\mathrm{Shv}(X,\mathcal{C})\rightarrow \mathrm{Shv}_{\mathcal{K}}(X,\mathcal{C}).\]
\end{prop}

Section 3 moves from sheaf theory to the more specialized realm of overconvergent sheaf theory. The main idea is to spell out the definitions and analogues of classical results on overconvergent sheaves of abelian groups in the setting of sheaves valued in stable \(\infty\)-categories. Probably the newest content in Section 3 is the definition of overconvergence of cosheaves: having defined overconvergence of sheaves, we define an overconvergent cosheaf \(\mathcal{G}\) to be a cosheaf such that \(\mathcal{G}^{op}\) is overconvergent. This is an easy definition, but it does play just the right role our Verdier duality result. Let us briefly survey the contents of the main subsections of Section 3. Subsection 3.1 essentially just defines overconvergent sheaves and cosheaves. Subsection 3.2 establishes essential results on the relation between overconvergent sheaves on a rigid analytic space and sheaves on the underlying topological space of the associated Berkovich space, via establishing a relation between overconvergent sheaves and wide open \(G\)-topology sheaves. The main result of this subsection is Proposition 3.7, whose main contents are summarized as follows:

\begin{prop}
Let \(X\) be a rigid analytic space. The restriction functor

\[\mathrm{Shv}(\mathrm{M}(X),\mathcal{C}) \rightarrow \mathrm{Shv}_{wo}(X,\mathcal{C})\]

along the full subcategory inclusion

\[\mathrm{N}(\mathcal{W}(X)) \subset \mathrm{N}(\mathrm{Opens}(\mathrm{M}(X)))\]

of wide opens of \(X\) into the poset category of opens of the Berkovich topological space \(\mathrm{M}(X)\) given by

\[W \mapsto \mathrm{M}(W)\]

is an equivalence. Further, the functor

\[\mathrm{Shv}_{wo}(X,\mathcal{C})\xrightarrow{\mathrm{LKE}} \mathrm{EShv}_{wo}(X,\mathcal{C})\xrightarrow{\mathrm{restr}} \mathrm{OverShv}_{\mathcal{K}}(X,\mathcal{C})\]

given by left Kan extension (into what we call \itshape extended \upshape wide open sheaves for obvious reasons) followed by restriction to the poset category of quasi-compact admissible opens of \(X\) induces an equivalence between the \(\infty\)-categories of wide open \(G\)-topology sheaves and that of overconvergent \(\mathcal{K}\)-sheaves. 
\end{prop}
This straightforwardly also allows us to demonstrate that the \(\infty\)-categories of overconvergent sheaves and corresponding overconvergent \(\mathcal{K}\)-sheaves are equivalent.
\par
We now move to the main section 4, where we demonstrate that passage to quasi-compact supports induces an equivalence of \(\infty\)-categories between the overconvergent sheaf and cosheaf categories. Subsection 4.1 builds a sort of sections with quasicompact supports functor that we call the Verdier duality functor on arbitrary (not just overconvergent) sheaves, which in subsequent sections is shown to yield an equivalence with overconvergence assumptions. The subsection 4.2 provides a more direct proof of this equivalence that very closely follows Lurie's proof for the case of ambient quasi-compact \(X\). Subsection 4.3 proves the result for general \(X\), though here, we find it necessary to make more of an appeal to Verdier duality for the topological space \(\mathrm{M}(X)\). Finally, Subsection 4.4 defines the functor of direct image with quasi-compact supports and explains how, for appropriate morphisms of rigid analytic spaces, the previous subsections are enough to provide us with a right adjoint to this functor.
\par
Last, we provide an Appendix on the various properties of the Berkovich topological space \(\mathrm{M}(X)\) associated to rigid analytic \(X\). The main point is to record results about the assignment \(U \mapsto \mathrm{M}(U)\) for admissible opens \(U\subset X\), such as the fact that, for \(U\) wide open, these provide a basis for the topology on \(\mathrm{M}(X)\), and compatibilities of this assignment with respect to various operations like taking appropriate complement, unions, and intersections.

\bigskip
\paragraph{\bfseries Notations and terminology}

Throughout this paper, we extensively use the theory of \(\infty\)-categories, and the only concrete model we use is that of quasicategories (simplicial sets fulfilling a weakened version of the Kan complex condition, corresponding to the generalization from \(\infty\)-groupoids to \(\infty\)-categories). Our convention will be that all categories are \(\infty\)-categories (through the nerve construction), so the term \itshape category \upshape should be read as \(\infty\)-category, and we will feel free to refer to 1-categories without expressly passing to the nerve, though this is implied. The main reference for what we use about quasicategories is Jacob Lurie's text \itshape Higher Topos Theory. \upshape By default, \(\mathcal{C}\) will denote a stable \(\infty\)-category with all small limits and colimits. 
\par
For our conventions on rigid geometry, let \(\mathbb{F}\) be a field complete with respect to a nontrivial nonarchimedean absolute value. We by default take \(X\) to be a \itshape separated, paracompact \upshape rigid analytic space over \(\mathbb{F}\), and by default, all ambient rigid analytic spaces satisfy this assumption. The term paracompact refers to the condition (*) in the work of Schneider and van der Put \itshape Points and topologies in rigid geometry \upshape in Proposition 4 on the section General Rigid Spaces. That is, the space admits an admissible covering by open affinoids, where each open affinoid meets only finitely many others nontrivially. In similar spirit, denote by \(\mathrm{An}_{\mathbb{F}}\) the category of separated, paracompact rigid analytic spaces over \(\mathbb{F}\) with the usual maps of rigid spaces. 
\par
 We also fix some notation regarding poset categories that appear in the forthcoming discussion of sheaf theory. As a preliminary remark, we often refer interchangeably to posets and the corresponding poset \(1\)-category. For \(T\) a topological space, denote by \(\mathrm{Opens}(T)\) the poset (category) of open subsets. Denote by \(\mathcal{U}(X)\) the poset (category) of admissible opens of the rigid analytic space \(X\), and similarly, \(\mathcal{W}(X)\) for the wide open admissible open subsets (so that the inclusion into \(X\) is partially proper), and \(\mathcal{K}(X)\) for the poset of quasi-compact admissible opens of \(X\). We will in general denote the elements of a poset \(\mathrm{P}\) lying above a given one \(p \in \mathrm{P}\) by \(\mathrm{P}_{p \subset}\), or \(\mathrm{P}_{\subset p}\) to denote elements of \(\mathrm{P}\) that are \(\subset p\). There are also generalizations of this notation, where \(\mathrm{Q}\subset\mathrm{P}\) is a subposet, and \(\mathrm{Q}_{p \subset}\) denotes elements of \(\mathrm{Q}\) lying above \(p\) in \(\mathrm{P}\). We will denote the poset of quasi-compact admissible opens that a given \(K\) is inner in (see later in the text for the definition) by \(\mathcal{K}_{K \subset\subset}(X)\).
\par
We also assume general familiarity with the notations employed in Lurie's texts \itshape Higher Algebra \upshape and \itshape Higher Topos Theory, \upshape as well as those employed in the works of Schneider and van der Put in the references. 

\bigskip
\paragraph{\textbf{Acknowledgements}}
This work is extremely indebted especially to the work of Jacob Lurie on Verdier duality, and we make extensive use not just of results, but also of the ideas and presentation. Also, the work of Peter Schneider and Marius van der Put on subjects such as analytic points, the relation between the topologies on rigid analytic, adic, and Berkovich spaces, as well as the general theory of overconvergent sheaves, have heavily influenced the author's understanding. Thanks to Damien Lejay for some helpful remarks on gluing cosheaves.
Finally, the author acknowledges use (modified as needed for his purposes) of TeX code sourced (to the best of his knowledge) at https://arxiv.org/abs/1601.00859 in both versions of this paper on the arXiv, and more generally claims no originality with respect to the TeX tools used in writing it.
\section{Some Sheaf Theory}

This section collects some sheaf theory in the rigid analytic setting: particularly, the sense in which sheaf theory on the full rigid space relates to behavior on quasi-compact admissible opens.

\begin{defin}
A \itshape \(\mathcal{K}\)-sheaf \upshape on a rigid analytic space \(X\) valued in a stable \(\infty\)-category \(\mathcal{C}\) with small limits and colimits is a functor \(\mathcal{F}: \mathrm{N}(\mathcal{K}(X))^{op} \rightarrow \mathcal{C}\) satisfying the following:

(i) \(\mathcal{F}(\emptyset)\) is final. 

(ii) For every pair of quasi-compact admissible open subspaces \(K,K'\) of \(X\), the diagram

\[
\begin{tikzcd}
\mathcal{F}(K \cup K') \arrow[d] \arrow[r]  & \mathcal{F}(K) \arrow[d]  \\
\mathcal{F}(K') \arrow[r] & \mathcal{F}(K \cap K')
\end{tikzcd}
\]

is a pullback square.

Denote the full subcategory of \(\mathrm{Fun}(\mathrm{N}(\mathcal{K}(X))^{op},\mathcal{C})\) consisting of \(\mathcal{C}\)-valued \(\mathcal{K}\)-sheaves by \(\mathrm{Shv}_{\mathcal{K}}(X,\mathcal{C})\). We can define \(\mathcal{C}\)-valued \(\mathcal{K}\)-cosheaves as functors \(\mathcal{G}:\mathrm{N}(\mathcal{K}(X))\rightarrow \mathcal{C}\) so \(\mathcal{G}^{op}\) is a \(\mathcal{K}\)-sheaf, and denote the corresponding \(\infty\)-category by \(\mathrm{CShv}_{\mathcal{K}}(X,\mathcal{C})\). 
\end{defin}

\begin{rem}
We will at times refer to \(\mathcal{K}\)-(co)sheaves also in the topological space setting, for instance when considering the underlying topological space of the Berkovich space of a rigid analytic space. 
\end{rem}

\begin{rem}
The separatedness hypothesis will (among the other useful roles it is playing) ensure that, when we take complements of compact \(K\) in \(X\), we get an admissible open set. 
\end{rem}

We will have a couple of lemmas first that are analogues of results in \itshape Higher Topos Theory \upshape 7.3.4.8-9. The first shows that \(\mathcal{K}\)-(co)sheaves, which a priori just satisfy a simple (co)descent condition for 2-element covers, satisfy another kind (more closely related to the usual analytic \(G\)-topology). The second shows that, by right Kan extension, we get a sheaf from a \(\mathcal{K}\)-sheaf (that is, the right Kan extension to all admissible opens from merely quasi-compact ones is a sheaf). This second lemma's analogue for cosheaves and left Kan extensions holds by the same arguments.

\begin{lem}
Let \(X\) be a rigid analytic space, \(\mathcal{C}\) a stable \(\infty\)-category with small limits and colimits. Consider an admissible covering by admissible opens \(\mathcal{U}\) of \(X\). Denote by \(\mathcal{K}_{\mathcal{U}}(X)\) the poset of quasi-compact admissible opens of \(X\) contained in at least one element of \(\mathcal{U}\). Let \(\mathcal{F} \) a \(\mathcal{K}\)-sheaf. Then, \(\mathcal{F}\) is a right Kan extension of its restriction to \(\mathrm{N}(\mathcal{K}_{\mathcal{U}}(X))^{op}\).
\end{lem}

\begin{proof}

 Note that \(\mathcal{U}\) satisfies the conclusion of the lemma if and only if, for any quasi-compact, admissible open \(K \subset X\), the cover \(\left\{K \cap U_i\right\}_{U_i \in \mathcal{U}}\) satisfies the conclusion of the lemma, replacing \(X\) by \(K\). So, we will assume for the remainder of the proof that \(X\) is quasi-compact, and that \(\mathcal{U}\) admits a finite refinement \(\mathcal{W}=\left\{K_1,...,K_n\right\}\) consisting of quasi-compact admissible opens. We want to show that, for any quasi-compact admissible open \(K \subset X\), the restriction of \(\mathcal{F}\) to \(\mathrm{N}(\mathcal{K}_{\mathcal{U}}(X) \cup \left\{K\right\})^{op}\) is a right Kan extension of the restriction to \(\mathrm{N}(\mathcal{K}_{\mathcal{U}}(X))^{op}\). We will induct on the size of the refinement \(\mathcal{W}\). The case \(n=0\) is the case where \(X\) is empty, whence we are done by \(\mathcal{F}(\emptyset)\) being final. The case of \(n=1\) is also obvious. Now, we perform induction, assuming that, when the refinement has size \(n-1\), \(\mathcal{U}\) is good, and assume the refinement now has size \(n\) and prove \(\mathcal{U}\) is good under this assumption.
\par
Note that, given a refinement of admissible cover \(\mathcal{U}'\) by some admissible cover \(\mathcal{U}\), that there is a sequence of subcategories

 \[\mathrm{N}(\mathcal{K}_{\mathcal{U}}(X)) \subset \mathrm{N}(\mathcal{K}_{\mathcal{U'}}(X)) \subset \mathrm{N}(\mathcal{K}(X)).\]

To carry out the inductive step, we notice the following fact. Suppose that, for every \(U' \in \mathcal{U}'\), \(\left\{U' \cap U\right\}_{U \in \mathcal{U}}\) is a good covering of \(U'\). Then, we note that we have \(\mathcal{U}'\) is good if and only if \(\mathcal{U}\) is.
Now suppose that we have shown the conclusion when the refinement \(\mathcal{W}\) has \(n-1\) elements. Consider the admissible open \(V=K_2 \cup \cdots \cup K_n\). We put \(\mathcal{U}':=\mathcal{U} \cup \left\{K_1\right\} \cup \left\{V\right\}\). This is an admissible covering with admissible refinement \(\mathcal{U}\). We wish to apply the fact to \(\mathcal{U}'\), so we proceed to show that, for any \(U' \in \mathcal{U}'\), \(\left\{U' \cap U\right\}_{U \in \mathcal{U}}\) is a good covering of \(U'\). This is easy for elements of \(\mathcal{U}\) and for \(K_1\). For \(V\), notice by the inductive hypothesis that, since \(\left\{W \cap V\right\}_{W \in \mathcal{U}}\) has an admissible refinement consisting of \(n-1\) quasi-compact admissible subsets, that this is a good cover of \(V\). So, we have reduced to demonstrating the lemma for \(\mathcal{U}'\) instead of \(\mathcal{U}\). We will now apply the above fact to further reduce to proving the lemma, replacing \(\mathcal{U}'\) by the cover \(\left\{K_1,V\right\}\). To do this, we must demonstrate that, for any \(U' \in \mathcal{U}'\), \(\left\{K_1 \cap U',V \cap U'\right\}\) is a good cover of \(U'\). To do this, we must demonstrate that \(\mathcal{F}|\mathrm{N}(\mathcal{K}(U'))^{op}\) is a right Kan extension of its restriction to \(\mathrm{N}(\mathcal{K}_{\left\{K_1\cap U',V\cap U'\right\}}(U'))^{op}\). To do this, consider some specific quasi-compact, admissible \(C \subset U'\). Put \(C_1:=K_1 \cap C,C_2:=V \cap C\). We must demonstrate that \(\mathcal{F}|\mathrm{N}(\mathcal{K}_{\left\{C_1,C_2\right\}}(C)\cup\left\{C\right\})^{op}\) is a right Kan extension of \(\mathcal{F}|\mathrm{N}(\mathcal{K}_{\left\{C_1,C_2\right\}}(C))^{op}\). This now follows from the fact that the inclusion

\[\mathrm{N}(C_1,C_2,C_1 \cap C_2) \subset \mathrm{N}(\mathcal{K}_{\left\{C_1,C_2\right\}}(C))\]
is cofinal. 
\par
In fact, we have completed the main proof: having reduced to the case where \(\mathcal{U}\) consists of two admissible open quasi-compact subspaces of quasi-compact \(X\), as all such covers are good.

\end{proof}

\begin{rem}
We note that the main significance of good coverings being \itshape admissible \upshape above is seen in the fact that we could reduce to the case of \(\mathcal{U}\) being finite (of size 2, in particular). Indeed, it is unsurprising that a rigid sheaf would be determined by its behavior on affinoids. However, since \(\mathcal{K}\)-sheaves only exhibit a limited gluing property with respect to covers with two compact subspaces, some further argument (such as the above lemma) is crucial. A concise summary of the above lemma is that any \itshape admissible \upshape covering is good. 
\end{rem}

Let us now continue on to the second required lemma. For a collection of admissible opens \(\mathcal{W}\) of \(X\), we denote by by \(\mathcal{K}_{\mathcal{W}}(X)\) the poset of quasi-compact admissible opens of \(X\) contained in at least one element of \(\mathcal{W}\).

\begin{lem}
Let \(X\) be a rigid analytic space, and \(\mathcal{C}\) be a stable \(\infty\)-category with all small limits and colimits. Let \(\mathcal{G}\) be a \(\mathcal{K}\)-sheaf on \(X\), and let \(\mathcal{F}: \mathrm{N}(\mathcal{U}(X))^{op} \rightarrow \mathcal{C}\) be a functor derived from \(\mathcal{G}\) by right Kan extension. Then, \(\mathcal{F}\) is a sheaf on the rigid space \(X\).
\end{lem}

\begin{proof} Let \(\mathcal{W}\) be a covering sieve of \(U \subset X\). We must demonstrate that \(\mathcal{F}|\mathrm{N}(\mathcal{W} \cup \left\{U\right\})^{op}\) is a right Kan extension of the restriction to \(\mathrm{N}(\mathcal{W})^{op}\). Note that \(\mathrm{N}(\mathcal{W}) \subset \mathrm{N}(\mathcal{K}_{\mathcal{W}}(X) \cup \mathcal{W})\) is cofinal, so it will suffice to demonstrate that the restriction of \(\mathcal{F}\) to \(\mathrm{N}(\mathcal{K}_{\mathcal{W}}(X) \cup \mathcal{W} \cup \left\{U\right\})^{op}\) is a right Kan extension of the corresponding restriction to \(\mathrm{N}(\mathcal{K}_{\mathcal{W}}(X) \cup \mathcal{W})^{op}\). We now consider the chain of subcategories

\[\mathrm{N}(\mathcal{K}_{\mathcal{W}}(X)) \subset \mathrm{N}(\mathcal{K}_{\mathcal{W}}(X) \cup \mathcal{W})  \subset \mathrm{N}(\mathcal{K}_{\mathcal{W}}(X) \cup \mathcal{W} \cup \left\{U\right\})\]

and note that, by [HTT] 4.3.2.8, it suffices to show the restriction of \(\mathcal{F}\) to \(\mathrm{N}(\mathcal{K}_{\mathcal{W}}(X) \cup \mathcal{W} \cup \left\{U\right\})^{op}\) is a right Kan extension of the restriction to \(\mathrm{N}(\mathcal{K}_{\mathcal{W}}(X))^{op}\).
\par
This is clear at every element except \(U\). Let us now consider \(\mathcal{K}_{\subset U}(X)\), the poset of quasi-compact admissible opens of \(X\) contained in \(U\). We now consider the chain of subcategories

\[\mathrm{N}(\mathcal{K}_{\mathcal{W}}(X))^{op} \subset \mathrm{N}(\mathcal{K}_{\subset U}(X))^{op} \subset \mathrm{N}(\mathcal{K}_{\subset U}(X) \cup \left\{U\right\})^{op}\]

Our desired result follows if we can demonstrate that the restriction of \(\mathcal{F}\) to the right is a right Kan extension of the restriction to the left. By [HTT] 4.3.2.8, if we can demonstrate the restriction to the center is a right Kan extension of the restriction to the left, and likewise for the right and center, we are done. The left to center case follows from the previous lemma, and the center to right case is a consequence of the very definition of \(\mathcal{F}\) by right Kan extension from \(\mathcal{G}\). 

\end{proof}

\begin{prop}
The restriction \(\mathrm{Shv}(X,\mathcal{C})\rightarrow \mathrm{Shv}_{\mathcal{K}}(X,\mathcal{C})\) is an equivalence.
\end{prop}

\begin{proof}
This can be done by identifying \(\mathrm{Shv}(X,\mathcal{C})\subset \mathrm{Fun}(\mathrm{N}(\mathcal{U}(X))^{op},\mathcal{C})\) with the full subcategory of functors which are a right Kan extension of their restriction to \(\mathrm{N}(\mathcal{K}(X))^{op}\), and so this restriction is a \(\mathcal{K}\)-sheaf. Any sheaf of course restricts to a \(\mathcal{K}\)-sheaf. We have also demonstrated that right Kan extension does send \(\mathcal{K}\)-sheaves to sheaves. So, to finish, we just note that any sheaf is a right Kan extension of its restriction to \(\mathrm{N}(\mathcal{K}(X))^{op}\). This follows, as for any \(U\) admissible open of \(X\), the collection of quasi-compact admissible opens contained in it is an admissible cover closed under finite intersection. 
\end{proof}

We now need another lemma that we collect for future use, to define the quasicompact supports functor that yields the Verdier duality equivalence of categories.

\begin{lem}
Let \(X\) be a rigid analytic space, and let \(K,K' \subset X\) be quasi-compact admissible open. Then, \(\left\{X \setminus K, X \setminus K'\right\}\) is an admissible covering of \(X \setminus K \cap K'\).
\end{lem}

\begin{proof}
The claim is that \(\left\{X\setminus K,X\setminus K'\right\}\) are an admissible covering of their union. It is clear these are both wide open, as \(\mathrm{M}(X) \setminus \mathrm{M}(K)=\mathrm{M}(X\setminus K)\) is open in \(\mathrm{M}(X)\), and similarly for \(K'\). We note that it will suffice to demonstrate that \(\left\{\mathrm{M}(X\setminus K), \mathrm{M}(X\setminus K')\right\}\) is a covering of \(\mathrm{M}(X\setminus K \cap K')\). However, from the compatibilities of \(\mathrm{M}(-)\), we see that

\[\mathrm{M}(X\setminus K) \cup \mathrm{M}(X\setminus K') \]

\[=(\mathrm{M}(X) \setminus \mathrm{M}(K)) \cup (\mathrm{M}(X) \setminus \mathrm{M}(K'))\]

\[= \mathrm{M}(X) \setminus (\mathrm{M}(K) \cap \mathrm{M}(K'))\]

\[=\mathrm{M}(X) \setminus \mathrm{M}(K \cap K')\]

\[ = \mathrm{M}(X\setminus K\cap K'),\]

as desired.

\end{proof}

\begin{defin}
Let \(X\) be a rigid analytic space (not necessarily quasi-compact). We consider \(\mathrm{N}(\mathcal{K}(X))\) as endowed with a Grothendieck topology, generated by finite coverings of quasi-compact admissible opens by other quasi-compact admissible opens of \(X\). That is, a covering sieve for an object \(K \in \mathcal{K}(X)\) is a full subcategory of \(\mathrm{N}(\mathcal{K}(X))_{/K}\) having the form \(\mathrm{N}(\mathcal{K}_{\mathcal{W}}(X))\times_{\mathrm{N}(\mathcal{K}(X))} \mathrm{N}(\mathcal{K}(X))_{/K}\) where \(\mathcal{W}\) is a covering of \(K\) by finitely many quasi-compact admissible opens. 

We call this the special \(G\)-topology, and denote the associated \(\infty\)-category of \(\mathcal{C}\)-valued sheaves by \(\mathrm{Shv}_{sp}(X,\mathcal{C})\).
\end{defin}

\begin{rem}
Our work on \(\mathcal{K}\)-sheaves actually demonstrates that the \(\infty\)-category \(\mathrm{Shv}_{\mathcal{K}}(X,\mathcal{C})\) is naturally equivalent with \(\mathrm{Shv}_{sp}(X,\mathcal{C})\). Indeed, this requires demonstrating that, for any covering of some \(K\) by \(\left\{K_1,...,K_n\right\}:=\mathcal{W}\), \(\mathcal{F}\in \mathrm{Shv}_{\mathcal{K}}(X,\mathcal{C})\) satisfies descent for the corresponding covering sieve, namely \(\mathcal{F}|\mathrm{N}(\mathcal{K}_{\mathcal{W}}(X) \cup \left\{K\right\})^{op}\) is a right Kan extension of \(\mathcal{F}|\mathrm{N}(\mathcal{K}_{\mathcal{W}}(X) )^{op}\). But we have actually demonstrated something stronger, where \(\mathcal{W}\) could be any admissible covering.

\end{rem}

\section{Basic Theory of Overconvergent (Co)sheaves}

The goal of this section is to define the basics notions of overconvergent (co)sheaf theory, and collect analogues of classical results in the \(\infty\)-categorical setting (particularly, the relation between overconvergent sheaves on the rigid analytic site and sheaves on the associated Berkovich topological space).

\subsection{Definitions}
\begin{defin}
Let \(X\) be a rigid analytic space. We say that a sheaf \(\mathcal{F} \in \mathrm{Shv}(X,\mathcal{C})\) is \itshape overconvergent \upshape if, for any quasi-compact admissible open \(K \subset X\), \(\mathcal{F}|\mathrm{N}(\mathcal{W}(X) \cup \left\{K\right\})^{op}\) is a left Kan extension of \(\mathcal{F}|\mathrm{N}(\mathcal{W}(X))^{op}\). Similarly, define a cosheaf \(\mathcal{G} \in \mathrm{CShv}(X,\mathcal{C})\) to be \itshape overconvergent \upshape if  \(\mathcal{G}|\mathrm{N}(\mathcal{W}(X) \cup \left\{K\right\})\) is a right Kan extension of \(\mathcal{G}|\mathrm{N}(\mathcal{W}(X))\). Denote by \(\mathrm{OverShv}(X,\mathcal{C})\) the full subcategory of \(\mathrm{Shv}(X,\mathcal{C})\) of overconvergent sheaves, and by \(\mathrm{OverCShv}(X,\mathcal{C})\) the full subcategory of \(\mathrm{CShv}(X,\mathcal{C})\) consisting of overconvergent cosheaves.
\end{defin}

\begin{rem}
One way to justify referring to the above cosheaves as \itshape overconvergent \upshape is that we can note \(\mathcal{G}: \mathrm{N}(\mathcal{U}(X)) \rightarrow \mathcal{C}\) is overconvergent as a cosheaf if and only if the opposite sheaf \(\mathcal{G}^{op}: \mathrm{N}(\mathcal{U}(X))^{op} \rightarrow \mathcal{C}^{op}\) is overconvergent (here, \(\mathcal{C}^{op}\) is the opposite stable \(\infty\)-category, also with all small limits and colimits, associated to \(\mathcal{C}\)).
\end{rem}

\begin{rem}
It is clear that, if \(\mathcal{F}\) on \(X\) is an overconvergent sheaf, \(\mathcal{F}|W\) is as well, for \(W \subset X\) wide open, due to cofinality of

\[\mathrm{N}(\mathcal{W}_{K \subset}(W))^{op} \subset \mathrm{N}(\mathcal{W}_{K \subset}(X))^{op}\]

where \(K \subset W\) is quasi-compact and admissible.
\end{rem}

We now define a notion needed to define the \(\mathcal{K}\)-sheaf version of overconvergence.

\begin{defin}
We say that one quasi-compact admissible open \(K\subset K' \subset X\) is \itshape inner \upshape in another \(K'\), written \(K \subset\subset K'\) if there exists some wide open \(W \subset X\) so that \(K \subset W \subset K'\).
\end{defin}

\begin{rem}
By general facts about the Berkovich topological space \(\mathrm{M}(X)\), this definition is equivalent to the demand that \(\mathrm{M}(K')\) is a neighborhood of \(\mathrm{M}(K)\). This is the characterization we find of the notion of \itshape inner \upshape in Schneider's \itshape Points of Rigid Analytic Varieties \upshape when the ambient \(X\) is affinoid. That is, \(\mathrm{M}(K)\) is inner in \(\mathrm{M}(K')\) in Schneider's sense if, for any \(x \in \mathrm{M}(K)\), there exists an affinoid neighborhood

\[x \in \mathrm{M}(Sp(B)) \subset \mathrm{M}(K')\]

of \(x\). If this is so, clearly \(\mathrm{M}(K')\) is a neighborhood of \(\mathrm{M}(K)\). Conversely, if \(\mathrm{M}(K')\) is a neighborhood of \(\mathrm{M}(K)\), then, letting \(x \in \mathrm{M}(K)\), since there is a fundamental system of compact neighborhoods of \(x\) arising from applying \(\mathrm{M}(-)\) to affinoid wide neighborhoods of \(x\) in \(X\), there is a neighborhood of \(x\) of form

\[x \in \mathrm{M}(Sp(B)) \subset \mathrm{M}(K') \subset \mathrm{M}(X).\]
\end{rem}

\begin{defin}
Let \(\mathcal{F}\) be a \(\mathcal{K}\)-sheaf on \(X\). We say \(\mathcal{F}\) is \itshape overconvergent \upshape if, for any quasi-compact admissible open \(K \subset X\), \(\mathcal{F}|\mathrm{N}(\mathcal{K}_{K \subset\subset}(X) \cup \left\{K\right\})^{op}\) is a left Kan extension of the restriction \(\mathcal{F}|\mathrm{N}(\mathcal{K}_{K \subset\subset}(X))^{op}\). An overconvergent \(\mathcal{K}\)-cosheaf is defined to be a \(\mathcal{K}\)-cosheaf \(\mathcal{G}\) so \(\mathcal{G}^{op}\) is an overconvergent \(\mathcal{K}\)-sheaf.  Denote by \(\mathrm{OverShv}_{\mathcal{K}}(X,\mathcal{C})\) the full subcategory of \(\mathrm{Shv}_{\mathcal{K}}(X,\mathcal{C})\) of overconvergent \(\mathcal{K}\)-sheaves, and by \(\mathrm{OverCShv}_{\mathcal{K}}(X,\mathcal{C})\) the full subcategory of \(\mathrm{CShv}_{\mathcal{K}}(X,\mathcal{C})\) consisting of overconvergent \(\mathcal{K}\)-cosheaves.
\end{defin}
We now discuss some of the motivation for the definition of an overconvergent sheaf that we have given. It is customary to define overconvergent sheaves on affinoid spaces, and then define them locally for global spaces. However, in the good situations, there is an equivalence between the notion of a global overconvergent sheaf and a sheaf for the underlying wide open \(G\)-topology, so one could morally also define overconvergent sheaves as ones arising from sheaves for wide open \(G\)-topologies. For the discussion, we recall that, for any full subcategory inclusion \(\mathcal{X}^0 \subset \mathcal{X}\), and any full subcategory \(\mathcal{B} \subset \mathrm{Fun}(\mathcal{X}^0,\mathcal{C})\), letting \(\mathcal{B}_{\mathrm{LKE}} \subset \mathrm{Fun}(\mathcal{X},\mathcal{C})\) denote the full subcategory of functors which are left Kan extensions of their restriction to \(\mathcal{X}^0\), which lies in \(\mathcal{B}\), restriction yields a trivial Kan fibration \(\mathcal{B}_{\mathrm{LKE}}\rightarrow \mathcal{B}\). An entirely analogous remark holds for right Kan extensions, namely we get a trivial Kan fibration \(\mathcal{B}_{\mathrm{RKE}}\rightarrow \mathcal{B}\). In both cases, these trivial Kan fibrations admit sections, which we denote \(\mathrm{LKE}\) and \(\mathrm{RKE}\) respectively (these in fact yield \(\infty\)-categorical left and right adjoints respectively to the restriction functors). Further, trivial Kan fibrations are closed under composition, which will come in handy when we combine the procedures of left and right Kan extension, as in the discussion of Verdier duality.
There is an inclusion \(i:\mathrm{N}(\mathcal{W}(X))^{op} \subset \mathrm{N}(\mathcal{U}(X))^{op}\), restriction along which induces a functor
\[i_*: \mathrm{PShv}(X,\mathcal{C})\rightarrow \mathrm{PShv}_{wo}(X,\mathcal{C})\]

(the latter denotes presheaves on the poset of admissible wide opens)
that preserves sheaves, and has a left adjoint (the presheaf pullback) given by left Kan extension \(i^*\). As is customary, one needs to compose \(i^*\) with sheafification \(\mathrm{L}\) to obtain the pullback of sheaves. An overconvergent sheaf, morally, should be a sheaf in the essential image of \(\mathrm{L} \circ i^*\circ j\) in \(\mathrm{Shv}(X,\mathcal{C})\), where \(j: \mathrm{Shv}_{wo}(X,\mathcal{C}) \subset \mathrm{PShv}_{wo}(X,\mathcal{C})\) is the inclusion of wide open sheaves into the corresponding presheaves. 
As we have seen earlier, precomposition along \(\mathrm{N}(\mathcal{K}(X))^{op} \subset \mathrm{N}(\mathcal{U}(X))^{op}\) induces an equivalence

\[\mathrm{Shv}(X,\mathcal{C})\rightarrow \mathrm{Shv}_{\mathcal{K}}(X,\mathcal{C}).\]

We make reference to the sheafification functor \(\mathrm{Fun}(\mathrm{N}(\mathcal{K}(X))^{op},\mathcal{C})\xrightarrow{\mathrm{L}} \mathrm{Shv}_{\mathcal{K}}(X,\mathcal{C})\), given by sheafification with respect to the earlier defined special Grothendieck topology. 
We claim the functor \(\mathrm{L} \circ i^*\) is equivalent to (that is, there exists a natural transformation that is an equivalence to) the composition (the last arrow is the obvious one: the inverse to restriction given by right Kan extension):

\[\mathrm{PShv}_{wo}(X,\mathcal{C}) \xrightarrow{i^*} \mathrm{PShv}(X,\mathcal{C}) \xrightarrow{\mathrm{restr}} \mathrm{Fun}(\mathrm{N}(\mathcal{K}(X))^{op},\mathcal{C})\xrightarrow{\mathrm{L}} \mathrm{Shv}_{\mathcal{K}}(X,\mathcal{C}) \cong \mathrm{Shv}(X,\mathcal{C})\]

which heuristically corresponds to applying \(i^*\) and then sheafifying the underlying \(\mathcal{K}\)-sheaf (note that this corresponds to sheafifying with respect to the special Grothendieck topology on quasi-compact admissibles). Indeed, this is because we have a sequence of fully faithful functors (the second category is the same as the third in the above sequence of arrows)

\[\mathrm{Shv}_{\mathcal{K}}(X,\mathcal{C}) \subset \mathrm{PShv}_{\mathcal{K}}(X,\mathcal{C}) \xrightarrow{\mathrm{RKE}} \mathrm{PShv}(X,\mathcal{C})\]

which can also be written

\[\mathrm{Shv}_{\mathcal{K}}(X,\mathcal{C})\xrightarrow{\mathrm{RKE}} \mathrm{Shv}(X,\mathcal{C}) \subset \mathrm{PShv}(X,\mathcal{C}).\]

We obtain a left adjoint to either of these compositions via composing left adjoints to the underlying functors involved in the composition.
For instance, \[\mathrm{PShv}(X,\mathcal{C}) \xrightarrow{\mathrm{L}} \mathrm{Shv}(X,\mathcal{C}) \xrightarrow{\mathrm{restr}} \mathrm{Shv}_{\mathcal{K}}(X,\mathcal{C})\]

or

\[\mathrm{PShv}(X,\mathcal{C}) \xrightarrow{\mathrm{restr}} \mathrm{PShv}_{\mathcal{K}}(X,\mathcal{C}) \xrightarrow{\mathrm{L}} \mathrm{Shv}_{\mathcal{K}}(X,\mathcal{C}).\]

Note that \(\mathrm{L}: \mathrm{PShv}(X,\mathcal{C}) \rightarrow \mathrm{Shv}(X,\mathcal{C})\) is equivalent to the same functor, after composing with \(\mathrm{Shv}(X,\mathcal{C}) \cong \mathrm{Shv}_{\mathcal{K}}(X,\mathcal{C}) \xrightarrow{\mathrm{RKE}} \mathrm{Shv}(X,\mathcal{C})\).  But this is equivalent to the composition

\[\mathrm{PShv}(X,\mathcal{C}) \xrightarrow{\mathrm{restr}} \mathrm{PShv}_{\mathcal{K}}(X,\mathcal{C}) \xrightarrow{\mathrm{L}} \mathrm{Shv}_{\mathcal{K}}(X,\mathcal{C}) \xrightarrow{\mathrm{RKE}} \mathrm{Shv}(X,\mathcal{C}).\]

The significance of this for us is that a functor in the essential image of

\[\mathrm{Shv}_{wo}(X,\mathcal{C}) \subset \mathrm{PShv}_{wo}(X,\mathcal{C}) \xrightarrow{i^*} \mathrm{PShv}(X,\mathcal{C})\xrightarrow{\mathrm{L}} \mathrm{Shv}(X,\mathcal{C})\]

is precisely one in the essential image of

\[\mathrm{Shv}_{wo}(X,\mathcal{C}) \subset \mathrm{PShv}_{wo}(X,\mathcal{C}) \xrightarrow{i^*} \mathrm{PShv}(X,\mathcal{C})\xrightarrow{\mathrm{restr}} \mathrm{PShv}_{\mathcal{K}}(X,\mathcal{C})\xrightarrow{\mathrm{L}} \mathrm{Shv}_{\mathcal{K}}(X,\mathcal{C}) \xrightarrow{\mathrm{RKE}} \mathrm{Shv}(X,\mathcal{C}).\]

Now, the main point for us is that any functor in the essential image of

\[\mathrm{Shv}_{wo}(X,\mathcal{C}) \subset \mathrm{PShv}_{wo}(X,\mathcal{C}) \xrightarrow{i^*} \mathrm{PShv}(X,\mathcal{C})\xrightarrow{\mathrm{restr}} \mathrm{PShv}_{\mathcal{K}}(X,\mathcal{C})\]

(as we will see) is already a \(\mathcal{K}\)-sheaf (which we take for granted now, though it is shown later), and an overconvergent sheaf should morally be precisely a functor in the essential image of the further right Kan extension

\[\mathrm{Shv}_{wo}(X,\mathcal{C}) \subset \mathrm{PShv}_{wo}(X,\mathcal{C}) \xrightarrow{i^*} \mathrm{PShv}(X,\mathcal{C})\xrightarrow{\mathrm{restr}} \mathrm{PShv}_{\mathcal{K}}(X,\mathcal{C})\xrightarrow{\mathrm{RKE}} \mathrm{PShv}(X,\mathcal{C}).\]

\subsection{Comparison of (Co)sheaf Theory on \(X\) and \(\mathrm{M}(X)\)}

It is well-known outside the \(\infty\)-categorical setting that the theory of overconvergent sheaves (of abelian groups) on an appropriate rigid analytic space (for example, Schneider and van der Put cover quasi-separated, paracompact spaces) is equivalent to that of sheaves on the associated Berkovich topological space. The purpose of this section is to collect the analogous results in our \(\infty\)-categorical setting for future use. We will often prove a result for sheaves, noting it implies a dual result for cosheaves, by considering for a cosheaf \(\mathcal{G}\) its opposite sheaf \(\mathcal{G}^{op}\). Let us now fix a bit of notation that will be used in the following. Let \(T\) be a locally compact, Hausdorff space. 
We will denote by \(\mathrm{EShv}(T,\mathcal{C})\) (standing for \itshape extended--to compact subspaces--sheaves) \upshape the \(\infty\)-category of sheaves satisfying the equivalent conditions of [HTT] 7.3.4.9.

\begin{prop}
Let \(X\) be a rigid analytic space. The restriction functor 

\[\mathrm{EShv}(\mathrm{M}(X),\mathcal{C}) \rightarrow \mathrm{Fun}(\mathrm{N}(\mathcal{W}(X) \cup \mathcal{K}(X))^{op},\mathcal{C})\]
 
induces an equivalence onto the essential image, the full subcategory of functors \(\mathcal{F}\) satisfying the following conditions, which we claim to be equivalent: (1) \(\mathcal{F}|\mathrm{N}(\mathcal{K}(X))^{op}\) is an overconvergent \(\mathcal{K}\)-sheaf, and \(\mathcal{F}\) is a right Kan extension of this restriction. (2) \(\mathcal{F}|\mathrm{N}(\mathcal{W}(X))^{op}\) is a sheaf for the wide open (partially proper) \(G\)-topology on \(X\), and \(\mathcal{F}\) is a left Kan extension of this restriction. Denoting by \(\mathrm{EShv}_{wo}(X,\mathcal{C})\) this full subcategory, this  produces a commutative diagram consisting entirely of equivalences:

\[
\begin{tikzcd}
\mathrm{Shv}_{\mathcal{K}}(\mathrm{M}(X),\mathcal{C}) \arrow[d]   & \arrow[l] \mathrm{EShv}(\mathrm{M}(X),\mathcal{C}) \arrow[d] \arrow[r]  & \mathrm{Shv}(\mathrm{M}(X),\mathcal{C}) \arrow[d]\\
\mathrm{OverShv}_{\mathcal{K}}(X,\mathcal{C})  & \arrow[l]  \mathrm{EShv}_{wo}(X,\mathcal{C}) \arrow[r]  & \mathrm{Shv}_{wo}(X,\mathcal{C})
\end{tikzcd}
\]

where all arrows are given by restriction, the vertical ones in particular given by restriction along functors defined using the assignments \(U \mapsto \mathrm{M}(U)\) on admissible opens.
\end{prop}

\begin{proof}
Let us suppose first that the conditions (1) and (2) claimed equivalent are in fact equivalent. We will demonstrate the remainder of the proposition first, returning to the conditions subsequently. To start with, we know we have a commutative diagram

\[
\begin{tikzcd}
\mathrm{Shv}_{\mathcal{K}}(\mathrm{M}(X),\mathcal{C}) \arrow[d]   & \arrow[l] \mathrm{EShv}(\mathrm{M}(X),\mathcal{C}) \arrow[d] \arrow[r]  & \mathrm{Shv}(\mathrm{M}(X),\mathcal{C}) \arrow[d]\\
\mathrm{Fun}(\mathrm{N}(\mathcal{K}(X))^{op},\mathcal{C})  & \arrow[l]  \mathrm{Fun}(\mathrm{N}(\mathcal{W}(X) \cup \mathcal{K}(X))^{op},\mathcal{C}) \arrow[r]  & \mathrm{Fun}(\mathrm{N}(\mathcal{W}(X))^{op},\mathcal{C})
\end{tikzcd}
\]

where all the arrows are given by restriction. To replace the bottom row with the appropriate full subcategories, we must conduct an analysis of where the various restrictions involved land. We certainly know the right-hand-side vertical arrow carries any object of \(\mathrm{Shv}(\mathrm{M}(X),\mathcal{C})\) to one of \(\mathrm{Shv}_{wo}(X,\mathcal{C})\) (see, for instance, the discussion in Appendix Corollary 5.11). Further, it is clear that the central arrow carries extended sheaves to functors whose restriction to \(\mathrm{N}(\mathcal{W}(X))^{op}\) is a sheaf for the wide open \(G\)-topology. Also, such a functor out of \(\mathrm{N}(\mathcal{K}(X) \cup \mathcal{W}(X))^{op}\) must be a left Kan extension of its restriction to \(\mathrm{N}(\mathcal{W}(X))^{op}\) since, for any quasi-compact admissible \(K \subset X\), the inclusion

\[\mathrm{N}(\mathcal{W}_{K \subset}(X))^{op} \subset \mathrm{N}(\mathrm{Opens}_{\mathrm{M}(K) \subset}(\mathrm{M}(X)))^{op}\]

is cofinal (Appendix Corollary 5.5). In any case, we now know that the functor

\[\mathrm{EShv}(\mathrm{M}(X),\mathcal{C}) \rightarrow \mathrm{Fun}(\mathrm{N}(\mathcal{W}(X) \cup \mathcal{K}(X))^{op}, \mathcal{C})\]

factors as \(\mathrm{EShv}(\mathrm{M}(X),\mathcal{C})\rightarrow \mathrm{EShv}_{wo}(X,\mathcal{C})\subset \mathrm{Fun}(\mathrm{N}(\mathcal{W}(X) \cup \mathcal{K}(X))^{op},\mathcal{C})\). So, we can rewrite our commutative diagram as having form

\[
\begin{tikzcd}
\mathrm{Shv}_{\mathcal{K}}(\mathrm{M}(X),\mathcal{C}) \arrow[d]   & \arrow[l] \mathrm{EShv}(\mathrm{M}(X),\mathcal{C}) \arrow[d] \arrow[r]  & \mathrm{Shv}(\mathrm{M}(X),\mathcal{C}) \arrow[d]\\
\mathrm{Fun}(\mathrm{N}(\mathcal{K}(X))^{op},\mathcal{C})  & \arrow[l] \mathrm{EShv}_{wo}(X,\mathcal{C}) \arrow[r]  &\mathrm{Shv}_{wo}(X,\mathcal{C})
\end{tikzcd}
\]

Note the bottom left horizontal arrow carries all objects to ones of \(\mathrm{OverShv}_{\mathcal{K}}(X,\mathcal{C})\). It is clear that the left vertical one carries all objects to those of \(\mathrm{Shv}_{\mathcal{K}}(X,\mathcal{C})\), but in fact, it also carries all objects to ones of \(\mathrm{OverShv}_{\mathcal{K}}(X,\mathcal{C})\), since for any quasi-compact, admissible \(K\) in \(X\), the inclusion

\[\mathrm{N}(\mathcal{K}_{K \subset\subset}(X))^{op} \subset \mathrm{N}(\mathcal{K}_{\mathrm{M}(K) \subset\subset}(\mathrm{M}(X))^{op}\]

is cofinal (Appendix Corollary 5.6). So our diagram is now
\[
\begin{tikzcd}
\mathrm{Shv}_{\mathcal{K}}(\mathrm{M}(X),\mathcal{C}) \arrow[d]   & \arrow[l] \mathrm{EShv}(\mathrm{M}(X),\mathcal{C}) \arrow[d] \arrow[r]  & \mathrm{Shv}(\mathrm{M}(X),\mathcal{C}) \arrow[d]\\
 \mathrm{OverShv}_{\mathcal{K}}(X,\mathcal{C}) & \arrow[l] \mathrm{EShv}_{wo}(X,\mathcal{C}) \arrow[r]  &\mathrm{Shv}_{wo}(X,\mathcal{C})
\end{tikzcd}
\]

Also, note the right vertical arrow is an equivalence. This is because, by Appendix Corollary 5.11, this restriction map to \(\mathrm{Shv}_{wo}(X,\mathcal{C})\) is identified with restriction to the \(\infty\)-category of sheaves for the basis of the topology on \(\mathrm{M}(X)\) of opens of form \(\mathrm{M}(W)\) with \(W\) wide open (this basis is closed under finite intersection, crucial for the \(\infty\)-categorical setting).
Since all horizontal arrows are known to be equivalences, this shows the vertical center one is as well, in turn implying the left vertical one is, too.
\par
We now turn to demonstrating the equivalence of conditions (1) and (2). First assume (2), and use the notation \(\mathrm{EShv}^2_{wo}(X,\mathcal{C})\) to denote the full subcategory of \(\mathrm{Fun}(\mathrm{N}(\mathcal{W}(X) \cup \mathcal{K}(X))^{op},\mathcal{C})\) satisfying just condition (2). Let \(\mathrm{F} \in \mathrm{EShv}^2_{wo}(X,\mathcal{C})\). Let us show the restriction of \(\mathrm{F}\) to \(\mathrm{N}(\mathcal{K}(X))^{op}\) is a \(\mathcal{K}\)-sheaf. Note that the natural restriction induces an equivalence

\[\mathrm{EShv}(\mathrm{M}(X),\mathcal{C})\rightarrow \mathrm{EShv}^2_{wo}(X,\mathcal{C}).\]

To begin, the fact that restriction (from \(\mathrm{M}(X) \) to \(X\)) carries extended sheaves on \(\mathrm{M}(X)\) to objects of \(\mathrm{EShv}^2_{wo}(X,\mathcal{C})\) follows from our earlier discussion. We can now demonstrate that

\[\mathrm{EShv}(\mathrm{M}(X),\mathcal{C})\rightarrow \mathrm{EShv}^2_{wo}(X,\mathcal{C})\]

is an equivalence by demonstrating both the composition and the last arrow of 

\[\mathrm{EShv}(\mathrm{M}(X),\mathcal{C})\rightarrow \mathrm{EShv}^2_{wo}(X,\mathcal{C})\rightarrow \mathrm{Shv}_{wo}(X,\mathcal{C}) \]

are. The last arrow clearly is, with inverse given by left Kan extension. The composition is just given by restriction to \(\mathrm{N}(\mathcal{W}(X))^{op}\). We can write it alternately as a composition 

\[\mathrm{EShv}(\mathrm{M}(X),\mathcal{C})\rightarrow \mathrm{Shv}(\mathrm{M}(X),\mathcal{C})\rightarrow \mathrm{Shv}_{wo}(X,\mathcal{C}) \]

and note the second arrow has an inverse given by right Kan extension, and the first one by left Kan extension. 
Thus, 

\[\mathrm{EShv}(\mathrm{M}(X),\mathcal{C})\rightarrow \mathrm{EShv}^2_{wo}(X,\mathcal{C})\]

is an equivalence. Hence, we see that for any \(\mathrm{F}\) in the target, there is a functor \(\mathrm{F}'\) in the source whose restriction to \(\mathrm{N}(\mathcal{W}(X) \cup \mathcal{K}(X))^{op}\) is equivalent to \(\mathrm{F}\). It is clear that the restriction of \(\mathrm{F}'\) to \(\mathrm{N}(\mathcal{K}(X))^{op}\) is a \(\mathcal{K}\)-sheaf, since we know that \(\mathrm{F}'|\mathrm{N}(\mathcal{K}(\mathrm{M}(X)))^{op}\) is one. This uses the compatibilities of the assignment \(K \mapsto \mathrm{M}(K)\) for \(K \subset X\) quasi-compact and admissible (preservation of finite unions and intersections). We now have to show that \(\mathrm{F}'|\mathrm{N}(\mathcal{K}(X))^{op}\) is overconvergent. 
\par
To do this, let \(K \subset X\) be quasi-compact and admissible.  Let \(\mathrm{V}=\mathcal{W}(X) \cup \mathcal{K}_{K \subset\subset}(X)\). Let \(\mathrm{V}'=\mathrm{V} \cup \left\{K\right\}\). We know \(\mathrm{F}'|\mathrm{N}(\mathrm{V})^{op}\) is a left Kan extension of \(\mathrm{F}'|\mathrm{N}(\mathcal{W}(X))^{op}\) and similarly, replacing \(\mathrm{V}\) with \(\mathrm{V}'\), since these are true of \(\mathrm{F}\). So, \(\mathrm{F}'|\mathrm{N}(\mathrm{V}')^{op}\) is a left Kan extension of \(\mathrm{F}'|\mathrm{N}(\mathrm{V})^{op}\). Now, note that the inclusion \(\mathrm{N}(\mathcal{K}_{K \subset\subset}(X))^{op} \subset \mathrm{N}(\mathcal{K}_{K \subset\subset}(X) \cup \mathcal{W}_{K \subset}(X))^{op}\) is cofinal (Appendix Corollary 5.7). Thus, we are done showing \(\mathrm{F}'|\mathrm{N}(\mathcal{K}(X))^{op}\) is an overconvergent \(\mathcal{K}\)-sheaf. Finally, we claim that \(\mathrm{F}'|\mathrm{N}(\mathcal{W}(X)\cup \mathcal{K}(X))^{op}\) is a right Kan extension of its restriction to \(\mathrm{F}'|\mathrm{N}(\mathcal{K}(X))^{op}\). To see this, note that, for any \(W \in \mathcal{W}(X)\), the functor

\[\mathrm{N}(\mathcal{K}_{\subset W}(X))\subset \mathrm{N}(\mathcal{K}_{\subset \mathrm{M}(W)}(\mathrm{M}(X)))\]

is cofinal (the latter denotes compact subspaces of \(\mathrm{M}(X)\) contained in \(\mathrm{M}(W)\)). Indeed, let \(C \in \mathcal{K}_{\subset \mathrm{M}(W)}(\mathrm{M}(X))\). We claim that, for any \(C \in \mathrm{N}(\mathcal{K}_{\subset \mathrm{M}(W)}(\mathrm{M}(X)))\),

\[\mathrm{N}(\left\{\mathrm{M}(K): K \in \mathcal{K}(X), \mathrm{M}(K) \subset \mathrm{M}(W)\right\})\times_{\mathrm{N}(\mathcal{K}_{\subset \mathrm{M}(W)}(\mathrm{M}(X)))} (\mathrm{N}(\mathcal{K}_{\subset \mathrm{M}(W)}(\mathrm{M}(X))))_{C/}\]

is weakly contractible. This is because, first of all, for any \(C \subset \mathrm{M}(W)\), there exists a compact neighborhood \(\mathrm{M}(K')\) of \(C\) contained in \(\mathrm{M}(W)\) with \(K'\subset X\) quasi-compact and admissible, by Appendix Remark 5.4. Second, the collection of such \(\mathrm{M}(K')\) lying between \(C\) and \(\mathrm{M}(W)\) is closed under finite intersection. This finishes the claim about weak contractibility. Now, simply note that the functor 

\[\mathrm{N}(\mathcal{K}_{\subset W}(X))\subset \mathrm{N}(\mathcal{K}_{\subset \mathrm{M}(W)}(\mathrm{M}(X)))\]

is a composition

\[\mathrm{N}(\mathcal{K}_{\subset W}(X))\cong \mathrm{N}(\left\{\mathrm{M}(K): K \in \mathcal{K}(X), \mathrm{M}(K) \subset \mathrm{M}(W)\right\}) \subset \mathrm{N}(\mathcal{K}_{\subset \mathrm{M}(W)}(\mathrm{M}(X)))\]

of two cofinal maps: an isomorphism and a cofinal inclusion. Thus, by [HTT] 4.1.1.3 (1) and (2), we are done demonstrating the claim of cofinality. From this, we see that \(\mathrm{F}'|\mathrm{N}(\mathcal{W}(X)\cup \mathcal{K}(X))^{op}\) is a right Kan extension of its restriction to \(\mathrm{F}'|\mathrm{N}(\mathcal{K}(X))^{op}\). Indeed, we know that that \(\mathrm{F}'\) is a right Kan extension of its restriction to \(\mathrm{N}(\mathcal{K}(\mathrm{M}(X)))^{op}\) at \(\mathrm{M}(W)\), and this is enough after the cofinality we demonstrated. 

Of course, since the restriction of \(\mathrm{F}'\) to \(\mathrm{N}(\mathcal{W}(X) \cup \mathcal{K}(X))^{op}\) is equivalent to \(\mathrm{F}\), by [HTT] 4.3.2.6, all that we have shown regarding \(\mathrm{F}'\) shows that (2) implies (1) for \(\mathrm{F}\). 
\par
It remains to demonstrate that the condition (1) implies (2). Assume (1), and consider some \(\mathrm{F} \in \mathrm{EShv}_{wo}^1(X,\mathcal{C})\) where the superscript 1 denotes that we are considering functors out of \(\mathrm{N}(\mathcal{W}(X) \cup \mathcal{K}(X))^{op}\) fulfilling (1). We must demonstrate that \(\mathrm{F}|\mathrm{N}(\mathcal{W}(X))^{op}\) is a wide open \(G\)-topologized sheaf. This is obvious, though, as \(\mathrm{F}\) is in the essential image of 

 \[\mathrm{Shv}_{\mathcal{K}}(X,\mathcal{C}) \xrightarrow{\mathrm{RKE}} \mathrm{Shv}(X,\mathcal{C})\xrightarrow{\mathrm{restr}} \mathrm{Fun}(\mathrm{N}(\mathcal{W}(X)\cup \mathcal{K}(X))^{op},\mathcal{C}).\]

 It remains to show that \(\mathrm{F}\) is a left Kan extension of \(\mathrm{F}|\mathrm{N}(\mathcal{W}(X))^{op}\). Let \(K \subset X\) be quasi-compact and admissible. We note that, by Appendix Corollary 5.7, there is a cofinal inclusion

\[\mathrm{N}(\mathcal{W}_{K \subset}(X))^{op} \subset \mathrm{N}(\mathcal{W}_{K \subset }(X) \cup \mathcal{K}_{K \subset\subset}(X))^{op}\]

and another one

\[\mathrm{N}(\mathcal{K}_{K \subset\subset}(X))^{op} \subset \mathrm{N}(\mathcal{W}_{K \subset }(X)\cup \mathcal{K}_{K \subset\subset}(X))^{op}.\]

From the first cofinal inclusion we note it will be enough to show that \(\mathrm{F}|\mathrm{N}(\mathrm{V}')^{op}\) is a left Kan extension of  \(\mathrm{F}|\mathrm{N}(\mathrm{V})^{op}\). This is clear except at \(\left\{K\right\}\). But, by the second cofinal inclusion, this follows since \(\mathrm{F}|\mathrm{N}(\mathcal{K}_{K \subset\subset}(X) \cup \left\{K\right\})^{op}\) is a left Kan extension of \(\mathrm{F}|\mathrm{N}(\mathcal{K}_{K \subset\subset}(X))^{op}\).
\par
Last, we remark that the composition

\[\mathrm{EShv}(\mathrm{M}(X),\mathcal{C})\xrightarrow{\mathrm{restr}} \mathrm{EShv}_{wo}(X,\mathcal{C}) \subset \mathrm{Fun}(\mathrm{N}(\mathcal{W}(X) \cup \mathcal{K}(X))^{op},\mathcal{C})\]

is one between an equivalence and a full, replete subcategory inclusion. Thus, the essential image is \( \mathrm{EShv}_{wo}(X,\mathcal{C})\).
\end{proof}

It follows from the above fairly straightforwardly that the notions of overconvergent sheaves and corresponding \(\mathcal{K}\)-sheaves are equivalent. 

\begin{cor}
There is a functor \(\mathrm{R}:\mathrm{OverShv}_{\mathcal{K}}(X,\mathcal{C}) \rightarrow \mathrm{OverShv}(X,\mathcal{C})\) furnished by right Kan extension, and it is an equivalence.
\end{cor}

\begin{proof}

This functor is obtained by restricting the functor \(\mathrm{RKE}:\mathrm{Shv}_{\mathcal{K}}(X,\mathcal{C}) \rightarrow \mathrm{Shv}(X,\mathcal{C})\) to the full subcategory \(\mathrm{OverShv}_{\mathcal{K}}(X,\mathcal{C})\). We first demonstrate that, for \(\mathcal{F} \in \mathrm{OverShv}_{\mathcal{K}}(X,\mathcal{C})\), \(\mathrm{RKE}(\mathcal{F})\) is overconvergent. However, this follows easily from the fact that the restriction of \(\mathrm{RKE}(\mathcal{F})\) to \(\mathrm{N}(\mathcal{W}(X) \cup \mathcal{K}(X))^{op}\) must, by (1) implying (2) in the lemma, be a left Kan extension of the restriction to \(\mathrm{N}(\mathcal{W}(X))^{op}\). Of course, any sheaf equivalent to \(\mathrm{RKE}(\mathcal{F})\) is also overconvergent, meaning any sheaf in the essential image of \(\mathrm{RKE}:\mathrm{OverShv}_{\mathcal{K}}(X,\mathcal{C}) \rightarrow \mathrm{Shv}(X,\mathcal{C})\) is overconvergent.
\par
To see the functor we have constructed is an equivalence, denote by \(\mathcal{R}\) the full subcategory of \(\mathrm{Fun}(\mathrm{N}(\mathcal{U}(X))^{op},\mathcal{C})\) consisting of functors which are right Kan extensions of their restriction to \(\mathrm{N}(\mathcal{K}(X))^{op}\), which is an overconvergent \(\mathcal{K}\)-sheaf. Restriction induces an equivalence \(\mathcal{R} \rightarrow \mathrm{OverShv}_{\mathcal{K}}(X,\mathcal{C})\). We are almost done. We will just show that a functor belongs to \(\mathcal{R}\) if and only if it is an overconvergent sheaf. Suppose \(\mathrm{F}\) is an overconvergent sheaf. Its restriction to \(\mathrm{N}(\mathcal{W}(X) \cup \mathcal{K}(X))^{op}\) then satisfies condition (2) of the lemma from above. This is because the restriction of any rigid analytic sheaf to \(\mathrm{N}(\mathcal{W}(X))^{op}\) is a sheaf for the wide open \(G\)-topology, and the left Kan extension condition follows from overconvergence. This implies (1), which entails that the restriction of \(\mathrm{F}\) to \(\mathrm{N}(\mathcal{K}(X))^{op}\) is an overconvergent \(\mathcal{K}\)-sheaf. Further, it is clear \(\mathrm{F}\) is a right Kan extension of this restriction, as this is true for any sheaf. This completes one direction. The other direction has already been shown above. 
\end{proof}

We are now in the position to complete the discussion from the end of the previous subsection. Let us note that, since the essential image of the composition

\[\mathrm{Shv}_{wo}(X,\mathcal{C})\xrightarrow{\mathrm{LKE}} \mathrm{EShv}_{wo}(X,\mathcal{C})\xrightarrow{\mathrm{restr}} \mathrm{OverShv}_{\mathcal{K}}(X,\mathcal{C})\xrightarrow{\mathrm{RKE}}\mathrm{OverShv}(X,\mathcal{C})\subset \mathrm{PShv}(X,\mathcal{C})\]

is identified with the essential image of

\[\mathrm{Shv}_{wo}(X,\mathcal{C}) \subset \mathrm{PShv}_{wo}(X,\mathcal{C})\xrightarrow{\mathrm{LKE}} \mathrm{PShv}(X,\mathcal{C})\xrightarrow{\mathrm{restr}} \mathrm{PShv}_{\mathcal{K}}(X,\mathcal{C})\xrightarrow{\mathrm{RKE}} \mathrm{PShv}(X,\mathcal{C}),\]

and since the first essential image is precisely the full subcategory of overconvergent sheaves, we note that, from our work at the end of the previous subsection, we have the following characterization of overconvergent sheaves:

\begin{prop}
The full subcategory \(\mathrm{OverShv}(X,\mathcal{C}) \subset \mathrm{Shv}(X,\mathcal{C})\) is the essential image of the composition

\[\mathrm{Shv}_{wo}(X,\mathcal{C}) \subset \mathrm{PShv}_{wo}(X,\mathcal{C}) \xrightarrow{i^*} \mathrm{PShv}(X,\mathcal{C}) \xrightarrow{\mathrm{L}} \mathrm{Shv}(X,\mathcal{C}).\]
\end{prop}

\section{The Verdier Duality Sheaf-Cosheaf Equivalence}

In this section, we provide a version of Lurie's Verdier duality theorem for overconvergent sheaves, showing that the Verdier functor sending an overconvergent sheaf to its associated cosheaf of (quasi-)compactly supported sections induces an equivalence between overconvergent sheaf and cosheaf categories. We remark that the Verdier dual of a sheaf is often supposed to be a \itshape sheaf,\upshape not a cosheaf, but for many natural choices of \(\mathcal{C}\), there is a natural way to associate to the cosheaf of (quasi-)compactly supported sections a dual sheaf. 
\subsection{Construction of the Verdier Duality Functor}
We let \(\mathrm{M}_X\), denoted \(\mathrm{M}\) when \(X\) is understood, be the poset of pairs \((i,S)\) where \(i=0,1,2\) and \(S \subset X\). If \(i=0\), we require \(S\) to be quasi-compact and admissible, and if \(i=2\), we require \(S \) to be of form \(X \setminus K\), where \(K\) is quasi-compact and admissible. The partial order will be given as follows: \((i,S) \leq (j,S')\) if (i) \(i=0,j=2\); (ii) \(S \subset S'\). We will denote by \(\mathrm{M}_0,\mathrm{M}_2,\mathrm{M}_2\) the posets where we only consider \(i=0,1,2\) respectively. Note that \(\mathrm{M}_0\) can be identified with \(\mathcal{K}(X)\) and \(\mathrm{M}_2\) with \(\mathcal{K}(X)^{op}\). We can also define a poset \(\mathrm{M}'\) similarly, where now we allow in addition pairs \((2,S)\) where \(S\) is the complement in \(X\) of some admissible open. We can identify \(\mathrm{M}_2'\) with \(\mathcal{U}(X)^{op}\). Denote by \(\mathcal{E}'\) the full subcategory of \( \mathrm{Fun}(\mathrm{N}(\mathrm{M}'),\mathcal{C})\) consisting of those functors \(\mathrm{F}\) satisfying (i) the restriction to \(\mathrm{N}(\mathrm{M}_2')\) is a sheaf (under the identification described earlier; we will often invoke such identifications implicitly); (ii) the restriction to \(\mathrm{N}(\mathrm{M}_1')\) is zero; (iii) the functor is a right Kan extension of its restriction to \(\mathrm{N}(\mathrm{M}_1' \cup \mathrm{M}_2')\). 
\par
Let us note that condition (ii) could be replaced by the condition: \(\mathrm{F}|\mathrm{N}(\mathrm{M}'_1\cup\mathrm{M}'_2)\) is a \itshape left \upshape Kan extension of \(\mathrm{F}|\mathrm{N}(\mathrm{M}'_2)\), since, for any \((1,S) \in \mathrm{M}'_1\), we know that \(\mathrm{N}(\mathrm{M}'_2)_{/(1,S)}\) is empty. This is an important observation, as this shows the restriction \(\mathcal{E}'\rightarrow \mathrm{Shv}(X,\mathcal{C})\) is a trivial Kan fibration, being the composition of two trivial Kan fibrations. It admits a section of form \(s:= \mathrm{RKE} \circ \mathrm{LKE}:\mathrm{Shv}(X,\mathcal{C})\rightarrow \mathcal{E}'.\)
\par
The Verdier duality functor will now be defined as follows. We will show that the restrictions of functors in \(\mathcal{E}'\) to \(\mathrm{N}(\mathrm{M}_0')\) are \(\mathcal{K}\)-cosheaves. Then, we just define the Verdier duality functor to be

\[\mathrm{Shv}(X,\mathcal{C})\xrightarrow{s} \mathcal{E}' \xrightarrow{\mathrm{restr}} \mathrm{CShv}_{\mathcal{K}}(X,\mathcal{C})\xrightarrow{\mathrm{LKE}} \mathrm{CShv}(X,\mathcal{C})\]

\par
It remains to demonstrate that restriction to \(\mathrm{N}(\mathrm{M}_0')\) induces a functor \(\mathcal{E}' \rightarrow \mathrm{CShv}_{\mathcal{K}}(X,\mathcal{C})\). To see this, let us consider two quasi-compact admissible opens \(K,K'\subset X\) for which we must demonstrate the \(\mathcal{K}\)-cosheaf codescent condition for restriction of \(\mathrm{F} \in \mathcal{E}'\) to \(\mathrm{N}(\mathrm{M}_0')\). Let us note that there are the following squares in \(\mathrm{N}(\mathrm{M}')\):

\[
\begin{tikzcd}
(1,K \cap K') \arrow[d] \arrow[r]  & (1,K')\arrow[d]  \\
(1,K) \arrow[r] & (1,K \cup K')
\end{tikzcd}
\]

\[
\begin{tikzcd}
(2,K \cap K') \arrow[d] \arrow[r]  & (2,K')\arrow[d]  \\
(2,K) \arrow[r] & (2,K \cup K')
\end{tikzcd}
\]

\[
\begin{tikzcd}
(2,\emptyset) \arrow[d] \arrow[r]  & (2,\emptyset) \arrow[d]  \\
(2,\emptyset) \arrow[r] & (2,\emptyset)
\end{tikzcd}
\]

and 

\[
\begin{tikzcd}
(0,K \cap K') \arrow[d] \arrow[r]  & (0,K')\arrow[d]  \\
(0,K) \arrow[r] & (0,K \cup K')
\end{tikzcd}
\]

There is an evident map from the first and third squares to the second, and a map from the last square to the first and third. We thus get a commutative square of squares. Composing with \(\mathrm{F}\), we obtain a diagram of squares in \(\mathcal{C}\). Note that, for every admissible, quasi-compact \(C \subset X\), the diagram \((2,\emptyset) \rightarrow (2,C) \leftarrow (1,C)\) is initial in \(\mathrm{N}(\mathrm{M}')_{(0,C)/} \times_{\mathrm{N}(\mathrm{M}')} \mathrm{N}(\mathrm{M}_1' \cup \mathrm{M}_2')\). Therefore, the condition that \(\mathrm{F}\) is a right Kan extension of the restriction to \(\mathrm{N}(\mathrm{M}_1' \cup \mathrm{M}_2')\) now yields that each

\[
\begin{tikzcd}
\mathrm{F}(0,C) \arrow[d] \arrow[r]  & \mathrm{F}(1,C)\arrow[d]  \\
\mathrm{F}(2,\emptyset) \arrow[r] & \mathrm{F}(2,C)
\end{tikzcd}
\]

is a pullback square. Note that the square
\[
\begin{tikzcd}
\mathrm{F}(0,K\cap K') \arrow[d] \arrow[r]  & \mathrm{F}(0,K)\arrow[d]  \\
\mathrm{F}(0,K') \arrow[r] & \mathrm{F}(0,K\cup K')
\end{tikzcd}
\]

is seen to be a fiber in the stable \(\infty\)-category \(\mathrm{Fun}(\Delta^1\times \Delta^1,\mathcal{C})\) of the map from the square

\[
\begin{tikzcd}
\mathrm{F}(2,\emptyset) \arrow[d] \arrow[r]  & \mathrm{F}(2,\emptyset)\arrow[d]  \\
\mathrm{F}(2,\emptyset) \arrow[r] & \mathrm{F}(2,\emptyset)
\end{tikzcd}
\]

to

\[
\begin{tikzcd}
\mathrm{F}(2,K\cap K') \arrow[d] \arrow[r]  & \mathrm{F}(2,K)\arrow[d]  \\
\mathrm{F}(2,K') \arrow[r] & \mathrm{F}(2,K\cup K')
\end{tikzcd}
\]

induced from applying \(\mathrm{F}\) to the map from the third to the second of the four squares in \(\mathrm{N}(\mathrm{M}')\) we mention above. 
Since this is a map of pullbacks (that the target square is a pullback follows from our earlier observation that \(\left\{X\setminus K',X\setminus K\right\}\) admissibly cover their union), the fiber is also a pullback square. However, this means it is also a pushout, which gives the \(\mathcal{K}\)-cosheaf codescent condition, as desired.
\newline
This demonstrates that the restriction to \(\mathrm{N}(\mathrm{M}_0')\) induces a functor \[\mathcal{E}' \rightarrow \mathrm{CShv}_{\mathcal{K}}(X,\mathcal{C})\].

\begin{rem}
The above demonstrates that the restriction of \(\mathrm{F}\) to \(\mathrm{N}(\mathrm{M}_0')\) is given informally by the assignment \(K \mapsto \Gamma_{K}(X,\mathcal{G})=fib (\Gamma(X,\mathcal{G}) \rightarrow \Gamma(X \setminus K,\mathcal{G}))\) on compact subspaces of \(X\), where \(\mathcal{G}\) is the sheaf determined by the restriction of \(\mathrm{F}\) to \(\mathrm{N}(\mathrm{M}_2')\), and that our constructed functor \(\mathrm{Shv}(X,\mathcal{C}) \rightarrow \mathrm{CShv}(X,\mathcal{C})\) is informally given by \(U \mapsto \mathrm{colim}_{K \in \mathcal{K}(U)} \Gamma_K(X,\mathcal{G})\).
\end{rem}

We now provide a slightly different construction of the duality functor. To do so, define a subposet \(\mathrm{M}_{ao}'\), defined exactly as \(\mathrm{M}'\) is, except that any pair \((1,S)\) must satisfy that \(S = X \setminus U\) for \(U\) admissible open. That is, \(S\) is not arbitrary. We can use the notation \(\mathrm{M}_{ao,i}'\) for \(i=0,1,2\) to denote the appropriate subposets of the various \(\mathrm{M}'_i\). Notice that any functor \(\mathrm{F}' \in \mathcal{E}'\) is actually a right Kan extension of its restriction to \(\mathrm{N}(\mathrm{M}_{ao,1}' \cup \mathrm{M}_{ao,2}')\). This follows easily if we can just demonstrate it at elements of \(\mathrm{M}_1' \setminus \mathrm{M}'_{ao,1}\). Let \((1,S) \in \mathrm{N}(\mathrm{M}_1' \setminus \mathrm{M}'_{ao,1})\). There exists \((1,S) \leq (1,X\setminus U)\) where \(U\subset X\) is admissible open, as we can always take \(U=\emptyset\). Note that we now have that the inclusion 

\[\mathrm{N}(\mathrm{M}'_{ao,1})_{(1,S)/}\subset \mathrm{N}(\mathrm{M}'_{ao,1}\cup \mathrm{M}'_{ao,2})_{(1,S)/}\]

is initial (the opposite is cofinal). Indeed, we have that the poset of \((1,S)\leq (i,X\setminus U)\) with \(U\) admissible in \(X\) is nonempty, and given such \((i,X\setminus U)\), we know \((1,S) \leq (1,X\setminus U) \leq (i,X\setminus U)\). Further, the \(X\setminus U'\) occurring in such pairs \((1,S)\leq (1,X \setminus U')\leq (i,X\setminus U)\) are stable under finite unions, since \(X\setminus U_1' \cup X\setminus U_2' = X\setminus (U_1' \cap U_2')\). Now, define a full subcategory

\[\mathcal{E}'_{ao}\subset \mathrm{Fun}(\mathrm{N}(\mathrm{M}'_{ao}),\mathcal{C})\]
to consist of those functors \(\mathrm{F}'\) whose restriction to \(\mathrm{N}(\mathrm{M}_{ao,2}')\) is a sheaf, whose restriction to \(\mathrm{N}(\mathrm{M}_{ao,1}')\) is zero, and which is a right Kan extension of its restriction to \( \mathrm{N}(\mathrm{M}_{ao,1}' \cup \mathrm{M}_{ao,2}')\). We have a restriction functor

\[\mathcal{E}'\rightarrow \mathcal{E}'_{ao}\]

along the inclusion

\[\mathrm{N}(\mathrm{M}_{ao}')\subset \mathrm{N}(\mathrm{M}').\]

Note that, strictly speaking, it was not necessary solely to get this restriction functor to demonstrate a functor \(\mathrm{F}' \in \mathcal{E}'\) is a right Kan extension of its restriction \(\mathrm{F}'|\mathrm{N}(\mathrm{M}'_{ao,1} \cup \mathrm{M}'_{ao,2})\) \itshape everywhere, \upshape but only at elements of \(\mathrm{M}'_2\) of form \((0,K)\), and this would follow readily from the fact that the diagram

\[(1,K)\rightarrow (2,K) \leftarrow (2,\emptyset)\]

is initial in \(\mathrm{N}(\mathrm{M}'_{ao,1}\cup\mathrm{M}'_{ao,2})_{(0,K)/}\).
Finally, define a poset \(\mathrm{M}^{top}_{ao}\) to consist of pairs \((i,S)\) with \(i=0,1,2\) and \(S \subset \mathrm{M}(X)\) of form \(\mathrm{M}(K)\) with \(K \subset X\) quasi-compact and admissible if \(i=0\), and of form \(\mathrm{M}(X) \setminus \mathrm{M}(U)\) with \(U \subset X\) admissible open otherwise. Define \((i,S) \leq (j,S')\) precisely if \(i=0,j=2\) or \(S \subset S'\). We now note that have an evident equivalence

\[\mathrm{N}(\mathrm{M}_{ao}')\cong \mathrm{N}(\mathrm{M}^{top}_{ao}).\]

This is given on objects by \((0,K) \mapsto (0,\mathrm{M}(K)),(i,X\setminus U) \mapsto (i,\mathrm{M}(X)\setminus \mathrm{M}(U))\) for \(i=1,2\). This identifies the objects of either poset. The morphisms are also identified in a natural way. Indeed, we know for two quasi-compact admissible opens \(K_1,K_2,\), we have \(K_1 \subset K_2\) if and only if \(\mathrm{M}(K_1) \subset \mathrm{M}(K_2)\). This gives the desired identification at index 0. When we restrict to indices 1 and 2, we note that, if \(U_1,U_2\) are admissible opens, \(\mathrm{M}(X) \setminus \mathrm{M}(U_2) \subset \mathrm{M}(X) \setminus \mathrm{M}(U_1)\) if and only if \(\mathrm{M}(U_1) \subset \mathrm{M}(U_2)\) if and only if \(U_1 \subset U_2\) if and only if \(X\setminus U_2 \subset X \setminus U_1\). Finally, when comparing elements of index 0 with those of index 1 or 2, we must check that, for \(K \subset X\) quasi-compact admissible open, and \(U\subset X\) admissible open, \(K \subset X \setminus U\) if and only if \(\mathrm{M}(K) \subset \mathrm{M}(X)\setminus \mathrm{M}(U)\). First suppose the former. Then, \(K \cap U = \emptyset\). Thus, \(\mathrm{M}(K) \cap \mathrm{M}(U)= \emptyset\). So, \(\mathrm{M}(K) \subset \mathrm{M}(X) \setminus \mathrm{M}(U)\), as desired for the first direction. Now, suppose this is true. Then, \(\mathrm{M}(K) \cap \mathrm{M}(U) = \emptyset\), so \(K \cap U= \emptyset\), so \(K \subset X \setminus U\), as desired. So we clearly have the desired identification of poset categories. 

Hence, we can just as well view \(\mathcal{E}'_{ao}\) as consisting of functors defined on a poset category derived from subsets of \(\mathrm{M}(X)\).
Furthermore, the restriction 

\[\mathcal{E}'\rightarrow \mathcal{E}'_{ao}\]

is an equivalence, since it fits into the following commutative diagram

\[
\begin{tikzcd}
\mathcal{E}'\arrow[d] \arrow[r]  & \mathrm{Shv}(X,\mathcal{C})\arrow[d]  \\
\mathcal{E}'_{ao} \arrow[r] & \mathrm{Shv}(X,\mathcal{C})
\end{tikzcd}
\]

where the right vertical arrow is the identity, and both horizontal arrows are equivalences (the natural restrictions to index 2). Thus, we can define the duality functor either using

\[\mathrm{Shv}(X,\mathcal{C}) \cong \mathcal{E}'\rightarrow \mathrm{CShv}_{\mathcal{K}}(X,\mathcal{C}) \xrightarrow{\mathrm{LKE}} \mathrm{CShv}(X,\mathcal{C})\]

or

\[\mathrm{Shv}(X,\mathcal{C}) \cong \mathcal{E}_{ao}'\rightarrow \mathrm{CShv}_{\mathcal{K}}(X,\mathcal{C}) \xrightarrow{\mathrm{LKE}} \mathrm{CShv}(X,\mathcal{C})\]
 
where the first equivalence \(\mathrm{Shv}(X,\mathcal{C}) \cong \mathcal{E}'\) is given by the earlier mentioned \(s\), and the equivalence \(\mathrm{Shv}(X,\mathcal{C}) \cong \mathcal{E}_{ao}'\) is given by \(\mathrm{Shv}(X,\mathcal{C})\xrightarrow{s} \mathcal{E}'\xrightarrow{\mathrm{restr}} \mathcal{E}'_{ao}\).

\subsection{Lurie-style Verdier Duality Equivalence for Quasi-compact spaces}
The purpose of this section is to prove that, for quasi-compact spaces, a relatively direct analogue of Lurie's proof of Verdier duality in the topological setting is available.

\begin{thm}
Let \(X\) be quasi-compact. The Verdier duality functor constructed in an earlier section restricts to an equivalence of \(\infty\)-categories \(\mathrm{OverShv}(X,\mathcal{C}) \cong \mathrm{OverCShv}(X,\mathcal{C})\).
\end{thm}

The main tool in proving the theorem will be the following proposition:

\begin{prop}
The following conditions for a functor \(\mathrm{F}: \mathrm{N}(\mathrm{M}) \rightarrow \mathcal{C}\) are equivalent: (i) The restriction to \(\mathrm{N}(\mathrm{M}_0)\) determines an overconvergent \(\mathcal{K}\)-cosheaf, that to \(\mathrm{N}(\mathrm{M}_1)\) is zero, and \(\mathrm{F}\) is a left Kan extension of its restriction to \(\mathrm{N}(\mathrm{M}_0 \cup \mathrm{M}_1)\). (ii) The restriction to \(\mathrm{N}(\mathrm{M}_2)\) determines an overconvergent \(\mathcal{K}\)-sheaf, that to \(\mathrm{N}(\mathrm{M}_1)\) is zero, and \(\mathrm{F}\) is a right Kan extension of its restriction to \(\mathrm{N}(\mathrm{M}_1 \cup \mathrm{M}_2)\). 
\end{prop}

Given the proposition (which we prove by proving the below lemma and that condition (ii) implies (i)), the theorem is easy to prove. We simply will consider the full subcategory \(\mathcal{E} \subset \mathrm{Fun}(\mathrm{N}(\mathrm{M}),\mathcal{C})\) satisfying the equivalent conditions of the proposition. The restrictions to \(\mathrm{N}(\mathrm{M}_0)\) and \(\mathrm{N}(\mathrm{M}_2)\) define equivalences to the overconvergent sheaf and cosheaf categories (since they do to the overconvergent \(\mathcal{K}\)-sheaf/cosheaf categories). In the proof of our theorem, we define a certain full subcategory \(\mathcal{D} \subset \mathrm{Fun}(\mathrm{N}(\mathrm{M}'),\mathcal{C})\), which fits into a commutative diagram (vertical arrows are full subcategory inclusions, with horizontal ones the obvious restrictions).

\[
\begin{tikzcd}
\mathrm{OverCShv}_{\mathcal{K}}(X,\mathcal{C}) \arrow[d]   & \arrow[l] \mathcal{D} \arrow[r]\arrow[d] &\mathrm{OverShv}(X,\mathcal{C}) \arrow[d]\\
\mathrm{CShv}_{\mathcal{K}}(X,\mathcal{C})  & \arrow[l] \mathcal{E}'\arrow[r] & \mathrm{Shv}(X,\mathcal{C}) 
\end{tikzcd}
\]       

as well as (with all arrows the obvious restrictions except the left vertical, the identity)

  \[
\begin{tikzcd}
\mathrm{OverCShv}_{\mathcal{K}}(X,\mathcal{C}) \arrow[d]   & \arrow[l] \mathcal{D} \arrow[r]\arrow[d] &\mathrm{OverShv}(X,\mathcal{C}) \arrow[d]\\
\mathrm{OverCShv}_{\mathcal{K}}(X,\mathcal{C})  & \arrow[l] \mathcal{E}\arrow[r] & \mathrm{OverShv}_{\mathcal{K}}(X,\mathcal{C}) 
\end{tikzcd}
\]   

with the bottom two horizontal arrows of the latter diagram and the top right horizontal arrow equivalences (once we know the proposition). Also, the right vertical arrow of this diagram is an equivalence by by Corollary 3.8 of Subsection 3.2. Thus, the central vertical arrow of the latter diagram is an equivalence. Now, all arrows but the top left horizontal one of this diagram are known to be equivalences, whence it is, as well. 
\par
From the characterization of \(\mathcal{D}\), it is clear that the section 

\[s: \mathrm{Shv}(X,\mathcal{C})\rightarrow \mathcal{E}'\]

restricts to a functor

\[s: \mathrm{OverShv}(X,\mathcal{C})\rightarrow \mathcal{D},\]

whence the Verdier duality functor of the earlier subsection induces

\[ \mathrm{OverShv}(X,\mathcal{C})\subset \mathrm{Shv}(X,\mathcal{C})\xrightarrow{s} \mathcal{E}'\xrightarrow{\mathrm{restr}} \mathrm{CShv}_{\mathcal{K}}(X,\mathcal{C}) \xrightarrow{\mathrm{LKE}} \mathrm{CShv}(X,\mathcal{C}),\]

which induces a composition consisting solely of equivalences:

\[ \mathrm{OverShv}(X,\mathcal{C}) \xrightarrow{s} \mathcal{D}\xrightarrow{\mathrm{restr}} \mathrm{OverCShv}_{\mathcal{K}}(X,\mathcal{C}) \xrightarrow{\mathrm{LKE}} \mathrm{OverCShv}(X,\mathcal{C}).\]

\par
We now prove the following lemma.

\begin{lem}
If condition (ii) implies condition (i) for any stable \(\infty\)-category \(\mathcal{C}\), then (i) implies (ii).
\end{lem} 

\begin{proof}

 Suppose (i) is true of \(\mathrm{F}\in \mathrm{Fun}(\mathrm{N}(\mathrm{M}),\mathcal{C})\). Then, note that \(\mathrm{G}:=\mathrm{F}^{op}:\mathrm{N}(\mathrm{M})^{op}\rightarrow \mathcal{C}^{op}\) satisfies:  \(\mathrm{G}|\mathrm{N}(\mathrm{M}_0)^{op}\) is an overconvergent \(\mathcal{K}\)-sheaf valued in \(\mathcal{C}^{op}\), \(\mathrm{G}|\mathrm{N}(\mathrm{M}_1)^{op}\) is zero, and \(\mathrm{G}\) is a right Kan extension of the restriction to \(\mathrm{N}(\mathrm{M}_0 \cup \mathrm{M}_1)^{op}\). We note this heuristically corresponds to \(\mathrm{G}\) fulfilling the condition (ii), except for \(\mathcal{C}^{op}\) rather than \(\mathcal{C}\). More precisely, there is a natural functor (actually, an equivalence)

\[\mathrm{Inv}: \mathrm{N}(\mathrm{M})\rightarrow \mathrm{N}(\mathrm{M}^{op})\cong \mathrm{N}(\mathrm{M})^{op}\]

induced by the assignment \((i,S) \mapsto (2-i,X \setminus S)\), and \(\mathrm{G}\) fulfills the stated conditions if and only if \(\mathrm{G} \circ \mathrm{Inv}\) fulfills (ii) with respect to \(\mathcal{C}^{op}\) instead of \(\mathcal{C}\). But, since (ii) implies (i) for \itshape any \upshape target stable \(\infty\)-category, this means \(\mathrm{G}\circ \mathrm{Inv}\) satisfies (i), again for \(\mathcal{C}^{op}\) the target \(\infty\)-category instead of \(\mathcal{C}\). But this precisely amounts to \(\mathrm{G}\) satisfying (ii)', namely that \(\mathrm{G}|\mathrm{N}(\mathrm{M}_2)^{op}\) is an overconvergent \(\mathcal{K}\)-cosheaf, the restriction to \(\mathrm{N}(\mathrm{M}_1)^{op}\) is zero, and \(\mathrm{G}\) is a left Kan extension of \(\mathrm{G}|\mathrm{N}(\mathrm{M}_1 \cup \mathrm{M}_2)^{op}\). This in turn implies that \(\mathrm{G}^{op}=\mathrm{F}\) satisfies (ii), as desired. 

\end{proof}

The remainder of this section consists of proving the key proposition 4.3 that implies the sheaf-cosheaf equivalence theorem. 
\begin{proof}
 We will, appealing to the lemma, just demonstrate that (ii) implies (i). So, suppose (ii) holds.  We can extend a functor \(\mathrm{F} \in \mathcal{E}\) to \(\mathrm{F}': \mathrm{N}(\mathrm{M}') \rightarrow \mathcal{C}\) as follows. Let \(\mathcal{D}\) denote the full subcategory of \(\mathrm{Fun}(\mathrm{N}(\mathrm{M}'),\mathcal{C})\) consisting of functors satisfying: (I) the restriction to \(\mathrm{N}(\mathrm{M}_2)\) is an overconvergent \(\mathcal{K}\)-sheaf; (II) the restriction to \(\mathrm{N}(\mathrm{M}_2')\) is a right Kan extension of that to \(\mathrm{N}(\mathrm{M}_2)\); (III) the restriction to \(\mathrm{N}(\mathrm{M}_1')\) is zero; and finally, (IV) the functor is a right Kan extension of the restriction to \(\mathrm{N}(\mathrm{M}_1' \cup \mathrm{M}_2')\). Note the relationship between \(\mathcal{E}'\) and \(\mathcal{D}\) is that, where functors in \(\mathcal{D}\) have restriction to \(\mathrm{N}(\mathrm{M'}_2)\) that is an \itshape overconvergent \upshape sheaf (which is what (I) and (II) together amount to, those of \(\mathcal{E}'\) merely have corresponding restriction an ordinary sheaf. 
\par
We now observe that we can reconstrue (II) as the condition (II'): the restriction to \(\mathrm{N}(\mathrm{M}_2' \cup \mathrm{M}_1')\) is a right Kan extension of that to \(\mathrm{N}(\mathrm{M}_2 \cup \mathrm{M}_1)\). This is because, for any \((2,S) \in \mathrm{M}_2'\), it is impossible for any \((1,S')\) to satisfy \((2,S) \leq (1,S')\), whence the inclusion 

\[\mathrm{N}(\mathrm{M}_2)\subset \mathrm{N}(\mathrm{M}_2 \cup \mathrm{M}_1)\] induces an equivalence

\[\mathrm{N}(\mathrm{M}_2)\times_{\mathrm{N}(\mathrm{M}')} \mathrm{N}(\mathrm{M}')_{(2,S)/}\subset \mathrm{N}(\mathrm{M}_2 \cup \mathrm{M}_1)\times_{\mathrm{N}(\mathrm{M}')} \mathrm{N}(\mathrm{M}')_{(2,S)/}.\]
\par
This being accomplished, we can now also reconstrue condition (IV) as (IV'): \(\mathrm{F}'\) is a right Kan extension of its restriction to \(\mathrm{N}(\mathrm{M}_1 \cup \mathrm{M}_2)\).
\par
Therefore, to produce an object \(\mathrm{F}' \in \mathcal{D}\) from \(\mathrm{F} \in \mathcal{E}\), we simply need to take the restriction of \(\mathrm{F}\) to \(\mathrm{N}(\mathrm{M}_1 \cup \mathrm{M}_2)\) and right Kan extend it to \(\mathrm{N}(\mathrm{M}')\). We do this, starting with some \(\mathrm{F}\) satisfying condition (ii). The resulting \(\mathrm{F}'\) has restriction to \(\mathrm{N}(\mathrm{M})\) equivalent to \(\mathrm{F}\). It will be enough to demonstrate that \(\mathrm{F}'|\mathrm{N}(\mathrm{M})\) satisfies (i). 
\par
Now, note that \(\mathcal{D}\) is naturally a full subcategory of \(\mathcal{E}'\). We have a commutative diagram

\[
\begin{tikzcd}
\mathrm{OverCShv}_{\mathcal{K}}(X,\mathcal{C}) \arrow[d]   & \arrow[l] \mathcal{D} \arrow[r]\arrow[d] &\mathrm{OverShv}(X,\mathcal{C}) \arrow[d]\\
\mathrm{CShv}_{\mathcal{K}}(X,\mathcal{C})  & \arrow[l] \mathcal{E}'\arrow[r] & \mathrm{Shv}(X,\mathcal{C}) 
\end{tikzcd}
\]       

with the vertical arrows all full subcategory inclusions, and the others all the natural restriction maps. The only nontrivial aspect of this is demonstrate that the restriction of a functor from \(\mathcal{D}\) to \(\mathrm{N}(\mathrm{M}_0')\) lands in \(\mathrm{OverCShv}_{\mathcal{K}}(X,\mathcal{C})\), not just \(\mathrm{CShv}_{\mathcal{K}}(X,\mathcal{C})\). In particular, our goal at present is to consider a functor \(\mathrm{F} \in \mathcal{E}\), produce from it a functor \(\mathrm{F}' \in \mathcal{D}\) as described earlier, and show first that \(\mathrm{F}'|\mathrm{N}(\mathrm{M}_0')\) is an overconvergent \(\mathcal{K}\)-cosheaf, and that \(\mathrm{F}'\) is a left Kan extension of its restriction to \(\mathrm{N}(\mathrm{M}_0 \cup \mathrm{M}_1)\) at any \((2,X\setminus K) \in \mathrm{N}(\mathrm{M}_2)\). We first demonstrate the first of these. To do this, note that we can consider the map of fiber sequences associated to quasi-compact admissible \(K \subset X\), where \(\mathcal{F}\) is the restriction of \(\mathrm{F}'\) to \(\mathrm{N}(\mathrm{M}_2')\) and \(\mathcal{G}\) is the restriction to \(\mathrm{N}(\mathrm{M}_0')\).

\[
\begin{tikzcd}
\mathcal{G}(K) \arrow[d] \arrow[r] & \underset{K' \in \mathrm{N}(\mathcal{K}_{K \subset\subset}(X))}{\mathrm{lim}} \mathcal{G}(K') \arrow[d] \\
\mathcal{F}(X ) \arrow[d] \arrow[r]  &  \underset{K' \in \mathrm{N}(\mathcal{K}_{K \subset\subset}(X))}{\mathrm{lim}}  \mathcal{F}(X) \arrow[d]  \\
\mathcal{F}(X \setminus K) \arrow[r] & \underset{K' \in \mathrm{N}(\mathcal{K}_{K \subset\subset}(X))}{\mathrm{lim}}  \mathcal{F}(X \setminus K')
\end{tikzcd}
\]
The middle horizontal arrow is an equivalence since its right-hand limit is taken with respect to a weakly contractible nerve, and the lower horizontal arrow is an equivalence because \(\mathrm{M}(X) \setminus \mathrm{M}(K)\) is covered by the collection (closed under finite intersection) \(\left\{\mathrm{M}(X) \setminus \mathrm{M}(K')\right\}_{K' \in \mathcal{K}_{K \subset\subset}(X)}\), since such \(\mathrm{M}(K')\) constitute a fundamental system of compact neighborhoods of \(\mathrm{M}(K)\). Thus, the first horizontal arrow is also an equivalence. Thus, \(\mathcal{G}\) is in fact overconvergent.
\par
It remains to demonstrate that \(\mathrm{F}\) is a left Kan extension of its restriction to \(\mathrm{N}(\mathrm{M}_0 \cup \mathrm{M}_1)\). Define \(\mathrm{M}''\) to be the poset consisting of objects of form \((i,S)\) where \(i=0,1\), and \(S \subset X\) is quasi-compact and admissible. We note the restriction of \(\mathrm{F}\) to \(\mathrm{N}(\mathrm{M}_0 \cup \mathrm{M}_1)\) is a left Kan extension of the corresponding restriction to \(\mathrm{N}(\mathrm{M}'')\). This is obvious at any \((i,S)\) with \(S\) quasi-compact and admissible. If \(S\) is general, the collection of quasi-compact, admissible \(K\subset S\) is non-empty, since we can always take \(K=\emptyset\). Now, note that 

\[\mathrm{N}(\mathrm{M}''_1)_{/(1,S)} \subset \mathrm{N}(\mathrm{M}'')_{/(1,S)}\]

is cofinal (the index 1 at left means we restrict to the subposet of \(\mathrm{M}''\) consisting of elements of index 1), since for any \((i,K)\leq (1,S)\), we have \((i,K)\leq (1,K)\leq (1,S)\), and for any two \(K \subset K_1,K_2 \subset S\), we know that \(K\subset K_1 \cup K_2 \subset S\) as well. So, we just need to show \(\mathrm{F}\) is a left Kan extension of \(\mathrm{F}|\mathrm{N}(\mathrm{M}'')\). To do this, we introduce an auxiliary subposet \(\mathrm{B}\) of \(\mathrm{M}'\)consisting of pairs \((2,X \setminus W)\) with \(W = X \setminus K\) wide open in \(X\) and \(K\) quasi-compact admissible open of \(X\). The appendix demonstrates that the collection of all these  \(\mathrm{M}(W)\)  is a basis for the topology on \(\mathrm{M}(X)\).

To demonstrate that  \(\mathrm{F}\) is a left Kan extension of \(\mathrm{F}|\mathrm{N}(\mathrm{M}'')\), we will demonstrate two claims: (2) \( \mathrm{F}'| \mathrm{N}( \mathrm{M}'' \cup  \mathrm{B}) \) is a left Kan extension of the restriction to \( \mathrm{N}( \mathrm{M}'')\); (1) \( \mathrm{F}'| \mathrm{N}( \mathrm{M}'' \cup  \mathrm{M}_2 \cup  \mathrm{B})\) is a left Kan extension of the restriction to \( \mathrm{N}(\mathrm{M}'' \cup  \mathrm{B})\).
\par
For claim (1), note that \(\mathrm{F}'|\mathrm{N}(\mathrm{M}_2 \cup \mathrm{B})\) is a left Kan extension of \(\mathrm{F}'|\mathrm{N}(\mathrm{B})\). This is because, for any \((2,X \setminus K) \in \mathrm{M}_2\) (\(K\subset X\) arbitrary quasi-compact and admissible in \(X\)), the evident inclusion

\[\mathrm{N}(\mathrm{B})_{/(2,X\setminus K)} \subset \mathrm{N}(\left\{(2,X \setminus W): W \in \mathcal{W}(X)\right\})_{/(2,X\setminus K)}\]

is cofinal. We note that the proof will now be complete if, for any \((2,X \setminus K) \in \mathrm{M}_2\), the inclusion

\[\mathrm{N}(\mathrm{B})_{/(2,X\setminus K)} \subset \mathrm{N}(\mathrm{B} \cup \mathrm{M}'')_{/(2,X\setminus K)}\]

is cofinal. This is seen by showing that, for any \((i,S) \in \mathrm{M}''\) so that \((i,S) \leq (2,X\setminus K)\), the poset

\[\left\{(2,X\setminus W): W \in \mathrm{B}, (i,S)\leq(2,X\setminus W)\leq (2,X\setminus K)\right\}\]

is both nonempty and stable under finite unions. Note that, if \(W_1=X\setminus C_1,W_2=X\setminus C_2\), the intersection \(W_1 \cap W_2 = X\setminus (C_1 \cup C_2)\), sow the stability is demonstrated. So, it remains to see the nonemptiness claim. That is, given \((i,S)\leq (2,X\setminus K)\) with \(i=0,1\) and \(S\) quasi-compact and admissible, we must demonstrate that there is some \((2,X\setminus W) \in \mathrm{B}\) sitting between \((i,S)\leq (2,X\setminus W)\leq (2,X\setminus K)\). The second of these two inequalities amounts to the claim that \(K \subset W\). The first is trivial if \(i=0\), but for \(i=1\), amounts to the claim that \(W \subset X \setminus S\). But now, we see this is a standard claim that, for two disjoint (that is, \(S \subset X\setminus K\)) quasi-compact and admissible \(S,K\), we can find some \(W\) containing \(K\) but completely disjoint from \(S\). We can accomplish this by finding corresponding \(\mathrm{M}(W)\) so that \(\mathrm{M}(W)\cap \mathrm{M}(S)=\emptyset\), and additionally, \(\mathrm{M}(K) \subset \mathrm{M}(W)\). This now follows from elementary facts about locally compact, Hausdorff spaces together with the collection of \(\mathrm{M}(W)\) having the property proved in the Appendix, Proposition 5.3.  
\par
It now remains to demonstrate the claim (2). We consider some \((2,X\setminus W) \in \mathrm{B}\), where \(W\) is of form \(X\setminus C\) with \(C\) admissible and quasi-compact. We see that there is a cofinal inclusion of the diagram \((0,X) \leftarrow (0,X\setminus W)\rightarrow (1,X\setminus W)\) into \(\mathrm{N}(\mathrm{M}'')_{/(2,X\setminus W)}\). Hence, we are reduced to demonstrating that the diagram

\[
\begin{tikzcd}
\mathrm{F}'(0,X \setminus W)  \arrow[r]\arrow[d] & \mathrm{F}'(1,X\setminus W) \arrow[d] \\
\mathrm{F}'(0,X) \arrow[r] & \mathrm{F}'(2,X\setminus W)\\
\end{tikzcd}
\]

is a pushout (equivalently, a pullback). To do this, we consider the larger diagram

\[
\begin{tikzcd}
\mathrm{F}'(0,C) \arrow[d] \arrow[r] & \mathrm{F}'(1,C) \arrow[d] \\
\mathrm{F}'(0,X) \arrow[r] \arrow[d] & \mathrm{F}'(2,C) \arrow[d] \\
\mathrm{F}'(2,\emptyset) \arrow[r] &\mathrm{F}'(2,C) 
\end{tikzcd}
\]

and note that the outer square is automatically a pullback. Thus, it will suffice to demonstrate that the left, bottom vertical arrow is an equivalence. To do this, simply consider that this edge fits into the following pullback diagram:

\[
\begin{tikzcd}
\mathrm{F}'(0,X) \arrow[d] \arrow[r] & \mathrm{F}'(1,X) \arrow[d] \\
\mathrm{F}'(2,\emptyset) \arrow[r] & \mathrm{F}'(2,X) \\
\end{tikzcd}
\]

where both of the right-side objects are zero. Thus, the left, vertical edge is indeed an equivalence, as desired. 
\end{proof}

\subsection{Verdier Duality Equivalence for General Spaces}

This subsection will build the Verdier duality equivalence for non-quasi-compact spaces. At least for the author, this case demands more heavy appeal to the Berkovich topological space of \(X\), and the proof is more indirect in a sense, appealing to Lurie's theorem for the topological space \(\mathrm{M}(X)\), because we do not have quite as convenient an analogue of the generating wide opens \(X \setminus K\) given by complements of quasi-compact admissibles in this setting. We begin by giving a construction of the Verdier duality functor more directly involving \(\mathrm{M}(X)\).
\par
Define a poset category \(\mathrm{P}'\) as follows. The objects are pairs \((i,S)\) with \(i=0,1,2\) and \(S\) a subset of \(\mathrm{M}(X)\) with the following restrictions: for \(i=0\), \(S\) is of form \(\mathrm{M}(K)\) with \(K\subset X\) admissible open and quasi-compact. If \(i=2\), we demand that \(S=\mathrm{M}(X)\setminus \mathrm{M}(U)\), where \(U \in \mathcal{W}(X) \cup \mathcal{K}(X)\). The partial order is defined as usual: we have \(\leq\) relations \((i,S)\leq (i',S')\) if either \(i=0,i'=2\) or \(S \subset S'\). Notice that we can identify \(\mathrm{P}_2'\) with the opposite to \(\mathcal{W}(X) \cup \mathcal{K}(X)\), since \(\mathrm{M}(X) \setminus \mathrm{M}(U_1) \subset \mathrm{M}(X) \setminus \mathrm{M}(U_2)\) precisely if \(U_2 \subset U_1\). Define \(\mathcal{D}\subset \mathrm{Fun}(\mathrm{N}(\mathrm{P'}),\mathcal{C})\) to be the full subcategory of functors \(\mathrm{F}'\) satisfying: the restriction to \(\mathrm{N}(\mathrm{P}'_2)\cong \mathrm{N}(\mathcal{W}(X) \cup \mathcal{K}(X))^{op}\) belongs to \(\mathrm{EShv}_{wo}(X,\mathcal{C})\); the restriction to \(\mathrm{N}(\mathrm{P}_1')\) is zero; finally, \(\mathrm{F}'\) is a right Kan extension of its restriction to \(\mathrm{N}(\mathrm{P}_1' \cup \mathrm{P}_2')\). 
We now obtain a Verdier duality functor

\[\mathrm{OverShv}(X,\mathcal{C})\cong \mathrm{EShv}_{wo}(X,\mathcal{C})\cong \mathcal{D} \xrightarrow{\mathrm{restr}} \mathrm{OverCShv}_{\mathcal{K}}(X,\mathcal{C}) \xrightarrow{\mathrm{LKE}} \mathrm{OverCShv}(X,\mathcal{C})\]

\begin{thm}
The above functor \(\Gamma_c(-)\) is an equivalence.
\end{thm}

\begin{proof}
We will refer to the constructions \(\mathcal{D},\mathcal{E}\) in Lurie's proof of Verdier duality, specifically for the locally compact Hausdorff space \(\mathrm{M}(X)\). We will use the notation \(\mathcal{D}_{\mathrm{M}(X)},\mathcal{E}_{\mathrm{M}(X)}\) defined as functors out of \(\mathrm{N}(\mathrm{M}_{\mathrm{M}(X)}')\) and \(\mathrm{N}(\mathrm{M}_{\mathrm{M}(X)})\) to avoid confusion with the notions we defined in the rigid analytic (as opposed to topological) setting. Note that restriction along

\[\mathrm{N}(\mathrm{P}') \subset \mathrm{N}(\mathrm{M}_{\mathrm{M}(X)}')\]

induces a functor

\[\mathcal{D}_{\mathrm{M}(X)} \rightarrow \mathcal{D}\]

Indeed, any functor \(\mathrm{F}'\) at left restricts to \(\mathrm{N}(\mathrm{P}_2')\) to be an extended wide open sheaf, and to zero at \(\mathrm{N}(\mathrm{P}_1')\). Further, the restriction to \(\mathrm{N}(\mathrm{P}')\) is a right Kan extension of that to \(\mathrm{N}(\mathrm{P}_1' \cup \mathrm{P}_2')\) because this condition amounts to, for \(K \subset X\) quasi-compact admissible, the square

\[
\begin{tikzcd}
\mathrm{F}'(0,\mathrm{M}(K)) \arrow[d] \arrow[r]  & \mathrm{F}'(1,\mathrm{M}(K))\arrow[d]  \\
\mathrm{F}'(2,\emptyset) \arrow[r] & \mathrm{F}'(2,\mathrm{M}(K))
\end{tikzcd}
\]

being a pullback, which follows from \(\mathrm{F}'\) being an object of \(\mathcal{D}_{\mathrm{M}(X)}\). (Note that,  here, we have used that \((2,\mathrm{M}(K)) \in \mathrm{P}'_2\), which uses that \(\mathrm{M}(K)=\mathrm{M}(X)\setminus (\mathrm{M}(X)\setminus \mathrm{M}(K)) = \mathrm{M}(X)\setminus \mathrm{M}(X\setminus K)\), since \(\mathrm{M}(X)\setminus \mathrm{M}(X\setminus K)\) is of form \(\mathrm{M}(X)\setminus \mathrm{M}(U)\), where \(U=X\setminus K\) is admissible open). This means that we have a commutative diagram

\[
\begin{tikzcd}
\mathrm{CShv}_{\mathcal{K}}(\mathrm{M}(X),\mathcal{C}) \arrow[d]   &\arrow[l] \arrow[d] \mathcal{D}_{\mathrm{M}(X)} \arrow[r] & \mathrm{EShv}(\mathrm{M}(X),\mathcal{C}) \arrow[d]  \\
\mathrm{OverCShv}_{\mathcal{K}}(X,\mathcal{C}) & \arrow[l]  \mathcal{D} \arrow[r] & \mathrm{EShv}_{wo}(X,\mathcal{C})
\end{tikzcd}
\]

Among the arrows belonging to the right square, all but the central vertical one are known to be equivalences. Thus, the central one is, too. Now, among all arrows belonging to the left square, all but the bottom left are known to be equivalences. Hence, the bottom left one is, as well. Obviously, left Kan extension induces an equivalence

\[\mathrm{OverCShv}_{\mathcal{K}}(X,\mathcal{C})\xrightarrow{\mathrm{LKE}} \mathrm{OverCShv}(X,\mathcal{C}).\]

So, we do indeed have that our original duality functor is an equivalence.

\end{proof}

\subsection{Extraordinary Direct and Inverse Images}

In this subsection, we explain how the Verdier duality equivalence we have built in terms of equivalences between categories of sheaves and cosheaves relates to the more traditional story involving existence of dualizing sheaves and extraordinary inverse images. First, we define the direct image with quasi-compact supports functor. We will denote by \((f_+,f^+)\) the adjoint pair of the direct image and pullback of rigid analytic \(\mathcal{C}\)-valued cosheaves respectively (note that the pullback is a \itshape right \upshape adjoint to direct image).

\begin{defin}
Let \(f:X\rightarrow Y\) be a map of rigid analytic spaces so that \(f_*\) preserves overconvergence of sheaves, and \(f_+\) preserves overconvergence of cosheaves. Define the \itshape direct image with quasi-compact supports \upshape functor, denoted 

\[f_!: \mathrm{OverShv}(X,\mathcal{C})\rightarrow \mathrm{OverCShv}(Y,\mathcal{C})\]

to be the composition

\[\mathrm{OverShv}(X,\mathcal{C})\xrightarrow{\Gamma_c(-)} \mathrm{OverCShv}(X,\mathcal{C}) \xrightarrow{f_+} \mathrm{OverCShv}(Y,\mathcal{C}) \xrightarrow{\Gamma_c(-)^{-1}} \mathrm{OverShv}(Y,\mathcal{C}).\]

\end{defin}

 We have the following theorem:

\begin{thm}
Let \(f: X \rightarrow Y\) be a map of rigid analytic spaces, so that the adjoint pair \((f_*,f^*)\) of functors between \(\mathrm{Shv}(X,\mathcal{C})\) and \(\mathrm{Shv}(Y,\mathcal{C})\) preserve overconvergence (and likewise for overconvergent cosheaves). Then, the functor \(f_!\) admits a right adjoint \(f^!\).

\end{thm}

\begin{proof}
First, we specify what the functor \(f^!: \mathrm{OverShv}(Y,\mathcal{C})\rightarrow \mathrm{OverShv}(X,\mathcal{C})\) is precisely, and then explain why it is the desired right adjoint. The functor itself is simply given by the composition

\[\mathrm{OverShv}(Y,\mathcal{C})\xrightarrow{\Gamma_c(-)} \mathrm{OverCShv}(Y,\mathcal{C}) \xrightarrow{f^+} \mathrm{OverCShv}(X,\mathcal{C}) \xrightarrow{\Gamma_c(-)^{-1}} \mathrm{OverShv}(X,\mathcal{C}).\]

It yields the desired right adjoint by the following. Let \(\mathcal{F} \in \mathrm{OverShv}(X,\mathcal{C}), \mathcal{G} \in \mathrm{OverShv}(Y,\mathcal{C})\). We now have natural identifications

\[\mathrm{Maps}_{\mathrm{OverShv(Y,\mathcal{C})}}(f_!\mathcal{F},\mathcal{G})\]

\[\cong \mathrm{Maps}_{\mathrm{OverShv(Y,\mathcal{C})}}(\Gamma_c^{-1}(f_+\Gamma_c(\mathcal{F})),\mathcal{G}) \]

\[\cong \mathrm{Maps}_{\mathrm{OverCShv(Y,\mathcal{C})}}(f_+\Gamma_c(\mathcal{F}),\Gamma_c(\mathcal{G}))\]
\[\cong \mathrm{Maps}_{\mathrm{OverCShv(X,\mathcal{C})}}(\Gamma_c(\mathcal{F}),f^+\Gamma_c(\mathcal{G}))\]

\[\cong \mathrm{Maps}_{\mathrm{OverShv(X,\mathcal{C})}}(\mathcal{F},\Gamma_c^{-1}(f^+\Gamma_c(\mathcal{G})))\]

\[\cong \mathrm{Maps}_{\mathrm{OverShv(X,\mathcal{C})}}(\mathcal{F},f^!\mathcal{G}).\]
This completes the proof.

\end{proof}

\section{Appendix on the Berkovich topological space \(\mathrm{M}(X)\)}

This section collects some elementary properties of the topological space \(\mathrm{M}(X)\) associated in Berkovich theory to a rigid analytic space \(X\), particularly the relations between \(X\) and \(\mathrm{M}(X)\). Our treatment of \(\mathrm{M}(X)\), when explicit about its construction, will refer to the one based on (maximal) filters built from the collection of admissible opens of \(X\). This section does not claim any originality and mainly just collects facts we need in a convenient place, as sometimes the author found a reference difficult to find. For this paper, we will always, for \(U \subset X\) admissible open, denote by \(\mathrm{M}(U)\) the subset of \(\mathrm{M}(X)\) given by those maximal filters containing \(U\). This is not to be confused with the slightly different denotation, where \(U\) is considered as a rigid analytic space of its own right, rather than as a subspace of \(X\), where we could take \(\mathrm{M}(U)\) to be the maximal filters for the admissible opens of \(U\).

\begin{prop}

 Let \(U,V\) be admissible opens of \(X\). Then, \(\mathrm{M}(U) \subset \mathrm{M}(V)\) if and only if \( U \subset V\). That is, \(\mathrm{M}(-)\) is compatible with inclusions of admissible opens.

\end{prop}

\begin{proof}
Suppose \(U \subset V\). Let \(\mathrm{F} \in \mathrm{M}(U)\). Since \(\mathrm{F}\) is closed under enlargement, \(\mathrm{F}\) contains \(V\). Now, for the other direction, suppose there is some \(x \in U\) so that \(x \notin V\). Denote by \(\mathrm{F}_x\) the collection of admissible opens \(U\) of \(X\) so \(x \in U\). We claim this is a maximal filter, first of all. It is a filter, as (all \(U_i\) being admissible opens): \(x \notin \emptyset\); \(x \in X\); \(x \in U_1,U_2\) implies \(x \in U_1 \cap U_2\); \(x \in U_1,U_1 \subset U_2\) implies \(x \in U_2\). To show it is maximal, let us note that, if not, there exists some maximal filter \(\mathrm{F}_x \subset \mathrm{G}_x\), where the containment is proper. Thus, there is an admissible open \(V_x \in \mathrm{G}_x\) so that \(x \notin V_x\). Note that, because \(\mathrm{G}_x\) is prime, there must be some admissible affinoid \(Sp(A) \subset V_x\) so that \(\mathrm{G}_x\) contains \(Sp(A)\). Note that, since \(x \notin Sp(A)\), \(x \in X\setminus Sp(A)\). Thus, \(X \setminus Sp(A) \in \mathrm{F}_x\subset \mathrm{G}_x\). Thus, \((X\setminus Sp(A)) \cap Sp(A) = \emptyset \in \mathrm{G}_x\). Thus, \(\mathrm{G}_x\) as described cannot exist, and \(\mathrm{F}_x\) is maximal. Now, the proof is easy to complete. It is clear that \(U \in \mathrm{F}_x\) but \(V \notin \mathrm{F}_x\), so \(\mathrm{M}(U) \) cannot be contained in \(\mathrm{M}(V)\), as desired. 
\end{proof}

We call the filter \(\mathrm{F}_x\) the \itshape neighborhood filter \upshape of \(x\).
One consequence of our result is that, if \(\mathrm{M}(U)=\mathrm{M}(V)\) with \(U,V\) admissible opens of \(X\), then \(U=V\). Indeed, this is because \(\mathrm{M}(U) \subset \mathrm{M}(V)\) implies \(U \subset V\), and \(\mathrm{M}(V) \subset \mathrm{M}(U)\) implies \(V \subset U\). A consequence of this consequence is that, if \(U\) is an admissible open so that \(\mathrm{M}(U) = \emptyset\), \(U=\emptyset\). This is because \(\mathrm{M}(\emptyset)=\emptyset\), since no filter can contain the empty subset, and so \(U = \emptyset\).

Another consequence of the result is that, for any rigid analytic space \(X\), defining \(\mathcal{U}^{top}(X)\) to be the poset of all subsets of \(\mathrm{M}(X)\) of form \(\mathrm{M}(U)\), there is a natural equivalence

\[\mathrm{N}(\mathcal{U}(X))\cong \mathrm{N}(\mathcal{U}^{top}(X))\]

given on objects by \(U \mapsto \mathrm{M}(U)\) and \(U \subset V\) goes to \(\mathrm{M}(U) \subset \mathrm{M}(V)\). We will frequently exploit this sort of functor given by an assignment \(U \mapsto \mathrm{M}(U)\) in various contexts (sometimes, in context of dealing only with wide admissible opens
 or quasi-compact admissible opens).

There are also two fairly simple properties we quickly record: (1) \(\mathrm{M}(-)\) sends admissible coverings of admissible opens to coverings. Indeed, given an admissible covering \(\left\{U_i\right\}_{i \in I}\) of \(U\), a maximal filter in \(\mathrm{M}(X)\) contained in \(\mathrm{M}(U)\) contains \(U\). Since a maximal filter is prime, it must must also contain some \(U_i\), which gives one direction. 
In the other direction, if a maximal filter contains some \(U_i\), by closure under enlargement, it must contain \(U\). 
(2) \(\mathrm{M}(-)\) is compatible finite intersections. This follows immediately from the closure of filters under finite intersection and enlargement. 
\par
We now turn our attention to quasi-compact admissible opens. Note, first, that if \(K \subset X\) is quasi-compact and admissible, \(\mathrm{M}(K) \subset \mathrm{M}(X)\) is a compact subspace. It is well known that \(\mathrm{M}(K_i)\) is for any open affinoid \(K_i\) of \(K\) (for instance, Schneider/van der Put prove that such \(\mathrm{M}(K_i)\) is homeomorphic to the Berkovich topological space associated to the corresponding affinoid algebra, which is compact and Hausdorff), and so this follows from the above compatibility with unions: if \(K=K_1\cup \cdots \cup K_n\), \(\mathrm{M}(K) = \mathrm{M}(K_1)\cup \cdots \cup \mathrm{M}(K_n)\)
\begin{prop}
Let \(K \subset X\) be a quasi-compact admissible open. Then, \(\mathrm{M}(X \setminus K) = \mathrm{M}(X) \setminus \mathrm{M}(K)\).
\end{prop}

\begin{proof}
The direction \(\mathrm{M}(X \setminus K) \subset \mathrm{M}(X) \setminus \mathrm{M}(K)\) is easy: if \(\mathrm{F}\) is a maximal filter containing \(X\setminus K\), if it in addition contained \(K\), it would have to contain the empty set, which shows \(\mathrm{F}\in \mathrm{M}(X)\setminus \mathrm{M}(K)\). For the reverse direction, let \(x \in \mathrm{M}(X) \setminus \mathrm{M}(K)\). This complement is an open subset of \(\mathrm{M}(X)\), so there exists some \(x \in \mathrm{M}(C) \subset \mathrm{M}(X) \setminus \mathrm{M}(K)\), with \(\mathrm{M}(C)\) a compact neighborhood of \(x\) and \(C\) quasi-compact and admissible. But this implies that \(C \cap K = \emptyset\), whence \(C \subset X \setminus K\). Since as a filter, \(x\) contains \(C\), it must contain \(X \setminus K\), completing the proof.
\end{proof}

Note that this immediately implies the analogous result for any admissible open \(U \subset X\), namely that, for any quasi-compact admissible \(K \subset U\subset X\), we have \(\mathrm{M}(U) \setminus \mathrm{M}(K)  = \mathrm{M}(U \setminus K)\). This is because the left-hand-side is \((\mathrm{M}(X) \setminus \mathrm{M}(K)) \cap \mathrm{M}(U) = \mathrm{M}(X\setminus K) \cap \mathrm{M}(U)=\mathrm{M}(U \setminus K)\), as desired.

Continuing in the spirit of studying the \(\mathrm{M}(W)\) with \(W \subset X\) wide open, we also have the following.

\begin{prop}
Let \(C \subset \mathrm{M}(X)\) be compact, and let \(C \subset \mathrm{M}(K')\) be a compact neighborhood, where \(K' \subset X\) is quasi-compact and admissible. Then, there exists a partially proper \(W \subset X\) so that \(C \subset \mathrm{M}(W) \subset \mathrm{M}(K') \subset \mathrm{M}(X)\).
\end{prop}

\begin{proof}
This proof essentially follows from elementary topological arguments, using that \(\mathrm{M}(X)\) is locally compact and Hausdorff, making extensive use of the fact that any point of \(\mathrm{M}(X)\) has a fundamental system of compact neighborhoods of form \(\mathrm{M}(K)\) for \(K \subset X\) quasi-compact, admissible. Since \(\mathrm{M}(K')\) is a compact neighborhood of \(C\), there exists an open neighborhood \(U \subset \mathrm{M}(K') \subset \mathrm{M}(X)\) of the topology of \(\mathrm{M}(X)\) containing \(C\). Note that \(\mathrm{M}(K') \setminus U\) is a compact subspace of \(\mathrm{M}(K')\). There exists some \(V \subset X\) a quasi-compact admissible open, so that \((\mathrm{M}(K') \setminus U) \subset \mathrm{M}(V)\) is a compact neighborhood in \(\mathrm{M}(X)\), but \(\mathrm{M}(V) \cap C = \emptyset\) (see the remark below). Put \(K:=V \cap K'\). Then, \(\mathrm{M}(K)\) is a compact neighborhood of \(\mathrm{M}(K') \setminus U\) with respect to the topology of \(\mathrm{M}(K')\), still not overlapping \(C\) nontrivially. But now, it is clear \(C \subset \mathrm{M}(K' \setminus K)=(\mathrm{M}(K') \setminus \mathrm{M}(K)) \subset U\). Further, we claim \(\mathrm{M}(K' \setminus K)\) is open in \(\mathrm{M}(X)\). This is because we know

\[\mathrm{M}(K') \setminus \mathrm{M}(K) = (\mathrm{M}(X) \setminus \mathrm{M}(K)) \cap \mathrm{M}(K')\]

and, intersecting both sides with \(U\), this gives 

\[\mathrm{M}(K') \setminus \mathrm{M}(K)=(\mathrm{M}(K') \setminus \mathrm{M}(K))\cap U = (\mathrm{M}(X) \setminus \mathrm{M}(K)) \cap \mathrm{M}(K') \cap U = (\mathrm{M}(X) \setminus \mathrm{M}(K)) \cap U\]

(the latter equality follows from \(U \subset \mathrm{M}(K')\)). The right-hand-side is an intersection to two opens of \(\mathrm{M}(X)\), so we are done with the claim.

This gives us that \(K' \setminus K \subset X\) is also partially proper. We thus can put \(W:=K' \setminus K\). Due to our assumption that \(\mathrm{M}(K) \cap C = \emptyset\), we in addition know that \(C \subset \mathrm{M}(K' \setminus K)\). 
\end{proof}

Our proof actually demonstrates that there exists a basis for the topology on \(\mathrm{M}(X)\) consisting of opens of form \(\mathrm{M}(W)\) with \(W\) wide open in \(X\) having form \(K' \setminus K\) with both \(K'\) and \(K\) quasi-compact admissible opens of \(X\). Indeed, given any open \(U \subset \mathrm{M}(X)\), and \(x \in U\), we know there is some quasi-compact admissible open \(K' \subset X\) so that \(x \in \mathrm{M}(K') \subset U\), where \(\mathrm{M}(K')\) is a compact neighborhood of \(x\). This is from the work of Schneider/van der Put. But our result/proof above show that there is some open of \(\mathrm{M}(X)\) of form \(\mathrm{M}(K'\setminus K)\) satisfying \(x \in \mathrm{M}(K'\setminus K') \subset \mathrm{M}(K') \subset U\). Further, given two such \(K_1'\setminus K_1,K_2'\setminus K_2\), note that their intersection

\[(K_1'\setminus K_1) \cap (K_2'\setminus K_2)=K_1' \cap K_2' \setminus ((K_1 \cup K_2)\cap (K_1' \cap K_2'))\]

is also wide open of the same form. 
\par
Furthermore, in the special case where \(X\) is quasi-compact, we get something further: taking \(K'=X\) and \(C \subset \mathrm{M}(X)\) some compact subspace, for any \(C \subset U \subset \mathrm{M}(X)\) with \(U\) open, there exists some \(C \subset \mathrm{M}(X) \setminus \mathrm{M}(K) \subset U\). That is, there is a basis consisting of opens of form \(\mathrm{M}(X\setminus K)\) for the topology on \(\mathrm{M}(X)\).
\par
We now collect some corollaries about cofinality that are a consequence of our work in this appendix. We first make the remark promised. 

\begin{rem}
Let \(C \subset \mathrm{M}(X)\) be a compact subspace. By standard topological arguments, \(C\) admits a fundamental system of compact neighborhoods of form \(\mathrm{M}(K')\), with \(K'\subset X\) quasicompact and admissible open. Suppose \(U\) is an open containing \(C\). For every point \(x \in C\): we can pick compact neighborhoods \(\mathrm{M}(K_x)\), containing open \(x \in U_x\), so that \(\mathrm{M}(K_x)\subset  U\). We note finitely many of the \(U_x\) have union containing \(C\) by compactness, so the corresponding \(\mathrm{M}(K_{x_1})\cup\cdots \cup \mathrm{M}(K_{x_n})=\mathrm{M}(K_{x_1}\cup \cdots \cup K_{x_n})\) is now contained in \(U\) and a compact neighborhood of \(C\) of the desired form. 
\par
This shows that, for any two compact subspaces \(C_1,C_2 \subset \mathrm{M}(X)\) with \(C_1 \cap C_2 = \emptyset\), we have some quasi-compact admissible open \(K \subset X\) so \(\mathrm{M}(K)\) is a compact neighborhood of \(C_1\) with \(\mathrm{M}(K) \cap C_2 = \emptyset\).  This is because, by elementary properties of locally compact, Hausdorff spaces, there is an open neigborhood \(U\) containing \(C_1\) that satisfies \(U \cap C_2=\emptyset\), and we can find a compact neighborhood \(C_1 \subset \mathrm{M}(K) \subset U\) of \(C_1\). This \(\mathrm{M}(K)\) by necessity satisfies \(\mathrm{M}(K) \cap C_2 = \emptyset\). This fact is used, for instance, in the proof of Proposition 5.3. The analogous result, stating there exists  \(C_1 \subset \mathrm{M}(W)\) with \(W\) wide open and \(C_2 \cap \mathrm{M}(W)=\emptyset\) is also true, using our work in this remark, along with 5.3. 
 
\end{rem}
\begin{cor}
The inclusion \(\mathrm{N}(\mathcal{W}_{K \subset}(X))^{op} \subset \mathrm{N}(\mathrm{Opens}_{\mathrm{M}(K)\subset}(\mathrm{M}(X)))^{op}\) is cofinal.
\end{cor}

\begin{proof}
This inclusion is written as the composition
 \[\mathrm{N}(\mathcal{W}_{K \subset}(X))^{op} \xrightarrow{\mathrm{top}} \mathrm{N}(\left\{\mathrm{M}(W) \in \mathrm{Opens}(\mathrm{M}(X)):W \in \mathcal{W}(X), \mathrm{M}(K) \subset \mathrm{M}(W)\right\})^{op} \]

\[\subset \mathrm{N}(\mathrm{Opens}_{\mathrm{M}(K)\subset}(\mathrm{M}(X)))^{op}\]

where the first map is an isomorphism of simplicial sets. We denote the middle simplicial set by \(S\). Note that the second map is cofinal. Indeed, for any open \(U \subset \mathrm{M}(X)\) containing \(\mathrm{M}(K)\), there exists by Appendix Remark 5.4 some compact neighborhood \(\mathrm{M}(K')\) of \(\mathrm{M}(K)\) and contained in \(U\), and by applying Appendix Proposition 5.3, we know there is also some \(\mathrm{M}(W)\) lying between \(\mathrm{M}(K')\) and \(\mathrm{M}(K)\), where \(W \subset X\) is wide open. Further, the collection of \(\mathrm{M}(W)\) lying between \(U\) and \(\mathrm{M}(K)\) is closed under finite intersection. Thus, for any \(U \in \mathrm{N}(\mathrm{Opens}_{\mathrm{M}(K)\subset}(\mathrm{M}(X)))^{op}\), we know that the simplicial set \[S \times_{\mathrm{N}(\mathrm{Opens}_{\mathrm{M}(K)\subset}(\mathrm{M}(X)))^{op}} \mathrm{N}(\mathrm{Opens}_{\mathrm{M}(K)\subset}(\mathrm{M}(X)))^{op})_{U/}\]

is weakly contractible. This shows the second arrow in our composition is cofinal, by [HTT] 4.1.3.1. We are now done by [HTT] 4.1.1.3 (1) and (2), as the functor we wish to be cofinal is a composition of an isomorphism (hence cofinal map) of simplicial sets and another cofinal map.
\end{proof}

\begin{cor}

The inclusion
\[\mathrm{N}(\mathcal{K}_{K \subset\subset}(X))^{op} \subset \mathrm{N}(\mathcal{K}_{\mathrm{M}(K) \subset\subset}(\mathrm{M}(X))^{op}\]

is cofinal.
\end{cor}

\begin{proof}
Our inclusion is given by the composition 
\[\mathrm{N}(\mathcal{K}_{K \subset\subset}(X))^{op}\xrightarrow{\mathrm{top}} \mathrm{N}(\left\{\mathrm{M}(K'):K' \in \mathcal{K}(X), \mathrm{M}(K) \subset\subset \mathrm{M}(K')\right\})^{op}\subset \mathrm{N}(\mathcal{K}_{\mathrm{M}(K) \subset\subset}(\mathrm{M}(X))^{op}, \]

where the first arrow is an isomorphism, and the second is cofinal. The former is true because the collection of quasi-compact admissibles \(K'\) so that \(K \subset\subset K'\) is the same as those quasi-compact admissible opens such that \(\mathrm{M}(K')\) is a compact neighborhood of \(\mathrm{M}(K)\); indeed, if \(K\subset\subset K'\), clearly \(\mathrm{M}(K)\subset\subset \mathrm{M}(K')\), but conversely, if \(\mathrm{M}(K)\subset\subset \mathrm{M}(K')\), then there is some \(\mathrm{M}(W)\) with \(W \in \mathcal{W}(X)\) so \(\mathrm{M}(K) \subset \mathrm{M}(W) \subset \mathrm{M}(K')\), whence the converse is shown. The latter is true because the \(\mathrm{M}(K')\) involved constitute a fundamental system of compact neighborhoods of \(\mathrm{M}(K)\) (see Remark 5.4 of this Appendix) closed under finite intersection, meaning that, denoting the middle simplicial set by \(S\), for any compact neighborhood \(C\) of \(\mathrm{M}(K)\), the simplicial set

\[S \times_{\mathrm{N}(\mathcal{K}_{\mathrm{M}(K) \subset\subset}(\mathrm{M}(X))^{op}} (\mathrm{N}(\mathcal{K}_{\mathrm{M}(K) \subset\subset}(\mathrm{M}(X))^{op})_{C/}\]

is weakly contractible. Now, we are done, by applying [HTT] 4.1.1.3 (1) and (2), since our composition is one of two cofinal functors (since isomorphisms of simplicial sets are cofinal).
\end{proof}

\begin{cor}
The inclusions 
\[\mathrm{N}(\mathcal{W}_{K \subset}(X))^{op} \subset \mathrm{N}(\mathcal{W}_{K \subset }(X) \cup \mathcal{K}_{K \subset\subset}(X))^{op}\]

and 

\[\mathrm{N}(\mathcal{K}_{K \subset\subset}(X))^{op} \subset \mathrm{N}(\mathcal{W}_{K \subset }(X)\cup \mathcal{K}_{K \subset\subset}(X))^{op}.\]

are cofinal.
\end{cor}

\begin{proof}
We begin with the first case. Let \(K \subset\subset K'\), where both are quasi-compact, admissible open in \(X\). We will show that the simplicial set 

\[\mathrm{N}(\mathcal{W}_{K \subset}(X))^{op}\times_{\mathrm{N}(\mathcal{W}_{K \subset }(X) \cup \mathcal{K}_{K \subset\subset}(X))^{op}} (\mathrm{N}(\mathcal{W}_{K \subset }(X) \cup \mathcal{K}_{K \subset\subset}(X))^{op})_{K'/}\]

is weakly contractible. Note that there exists \(W\in \mathcal{W}(X)\) so that \(K \subset W \subset K'\), since that is the definition of \(K \subset \subset K'\). Further, the collection of all \(W\) lying between \(K\) and \(K'\) is closed under finite intersection, whence we are done demonstrating weak contractibility, and the proof is finished. (We can also apply [HTT] 4.1.3.1 to show a weak contractibility claim at \(W \in \mathcal{W}_{K\subset}(X)\), not just at the various \(K \subset\subset K'\): this case is easy/omitted.)
\par
The second case will be handled similarly. Let \(K \subset W\), where \(W \in \mathcal{W}(X)\). We will show that 

\[(\mathrm{N}(\mathcal{K}_{K \subset\subset}(X))^{op} \times_{(\mathrm{N}(\mathcal{W}_{K \subset }(X)\cup \mathcal{K}_{K \subset\subset}(X))^{op}}(\mathrm{N}(\mathcal{W}_{K \subset }(X)\cup \mathcal{K}_{K \subset\subset}(X))^{op})_{W/}\]

is weakly contractible. Inded, there exists \(K' \in \mathcal{K}_{K\subset\subset}(X)\) such that \(K' \subset W\), since there exists \(K' \in \mathcal{K}(X)\) such that \(\mathrm{M}(K')\) is a compact neighborhood of \(\mathrm{M}(K)\) contained in \(\mathrm{M}(W)\).
\end{proof}

We commence with a result that implies that sheaves for the wide open \(G\)-topology on \(X\) can be identified with sheaves for the basis of \(\mathrm{M}(X)\) consisting of all \(\mathrm{M}(W)\) with \(W \subset X\) wide open. We record an elementary lemma first. 

\begin{lem}
Let \(X\) be a rigid analytic space. A cover \(\left\{U_i\right\}_{i \in I}=: \mathcal{U}\) of \(X\) by admissible opens of \(X\) is an admissible covering if and only if, for any open affinoid \(Sp(A)\subset X\), the pullback of \(\mathcal{U}\) to \(Sp(A)\) admits a finite refinement of open affinoids. 
\end{lem}

\begin{proof}
First, we note that a covering of some affinoid space \(Sp(A)\) by admissible opens is an admissible covering if and only if it admits a finite refinement consisting of open affinoids. Indeed, given a covering with such a refinement, since such a refinement is automatically admissible (see, for instance, Brian Conrad's \itshape Several Approaches to Nonarchimedean Geometry \upshape Example 2.2.7), we have the original covering is admissible. The other direction is obvious: by definition of the \(G\)-topology, an admissible cover must admit such a refinement. 
\par
So, we are reduced to proving the following statement: a cover \(\left\{U_i\right\}_{i \in I}=: \mathcal{U}\) of \(X\)  by admissible opens is an admissible covering if and only if, for any open affinoid \(Sp(A)\subset X\), the pullback of \(\mathcal{U}\) to \(Sp(A)\) is admissible. Since the collection of all open affinoids contained in \(X\) constitutes an admissible covering of \(X\), this is a consequence of the following general assertion: given some admissible covering \(\mathcal{W}=\left\{W_k\right\}_{k \in K}\) of \(X\), our cover \(\mathcal{U}\) is admissible if and only if its intersection with each element of \(\mathcal{W}\) is. Indeed, suppose each such intersection is admissible. The collection of all \(\left\{W_k\cap U_i\right\}_{i\in I,k \in K}\) is then an admissible covering of \(X\). Further, this is an admissible refinement of \(\mathcal{U}\), since each \(W_k \cap U_i\subset U_i\) is contained in some element of \(\mathcal{U}\). So, \(\mathcal{U}\) is also admissible. Conversely, suppose \(\mathcal{U}\) is admissible. Then, clearly \(\mathcal{U} \cap W_k\) is admissible for each \(W_k \in \mathcal{W}\), as desired. 
\end{proof}

\begin{lem}
Let \(\mathcal{U}:=\left\{U_i\right\}_{i \in I}\) be a collection of admissible opens of \(X\), and let \(U\) be another so that \(\left\{\mathrm{M}(U_i)\right\}_{i \in I}\) covers \(\mathrm{M}(U)\). Then, \(\left\{U_i\right\}_{i \in I}\) covers \(U\). 
\end{lem}

\begin{proof}
Let \(x \in U\). The neighborhood filter \(\mathrm{F}_x \in \mathrm{M}(U)\), whence by hypothesis, it is contained in some \(\mathrm{M}(U_i)\). Hence, \(\mathrm{F}_x\) contains \(U_i\) as a filter. By definition, this means \(x \in U_i\), so the proof is complete.
\end{proof}

\begin{prop}
Let \(\left\{W_i\right\}_{i \in I} =: \mathcal{W}\) be a collection of wide opens of \(X\), and consider wide open \(W\). Then, \(\mathcal{W}\) is an admissible covering of \(W\) \itshape if and only if \upshape \(\mathrm{M}(\mathcal{W}):=\left\{\mathrm{M}(W_i)\right\}_{i \in I}\) is a covering of \(\mathrm{M}(W)\). 
\end{prop}

\begin{proof}
One direction is obvious: if \(\mathcal{W}\) is an admissible cover of \(W\), then clearly \(\mathrm{M}(\mathcal{W})\) covers \(\mathrm{M}(W)\) (in fact, this holds for any admissible covering, not just wide open admissible coverings). Conversely, suppose \(\mathrm{M}(\mathcal{W})\) covers \(\mathrm{M}(W)\). From the above lemma, we get that \(\mathcal{W}\) covers \(W\), so we just need to show the covering is admissible. Let \(Sp(A):=Y \subset W\) be admissible open affinoid. Note that, for any \(x \in \mathrm{M}(Y)\), there exists some \(W_i\) so that \(x \in \mathrm{M}(W_i)\). Applying Appendix Remark 5.4, pick a quasi-compact admissible open \(V_x\subset X\) so that \(\mathrm{M}(V_x)\) is a compact neighborhood of \(x\) with \(x \in \mathrm{M}(V_x) \subset \mathrm{M}(W_i)\). By compactness of \(\mathrm{M}(Y)\), there is a finite subcollection \(\mathrm{M}(V_{x_1}),...,\mathrm{M}(V_{x_n})\) whose union contains \(\mathrm{M}(Y)\). This implies that

\[\mathrm{M}(Y) \subset \mathrm{M}(V_{x_1}\cup \cdots \cup V_{x_n}),\]

which in turn implies that \(Y \subset V_{x_1}\cup \cdots \cup V_{x_n}\). Clearly, each \(V_{x_i}\) is, by construction, contained in some \(W_k\). So, we have built a finite refinement of \(\mathcal{W}\) consisting of quasi-compact admissible opens of \(X\) whose union contains \(Sp(A)\). Thus, the pullback of \(\mathcal{W}\) to \(Sp(A)\) admits a finite refinement \(\left\{Y\cap V_{x_1},...,Y\cap V_{x_n}\right\}\) of quasi-compact admissible opens (and hence affinoid admissible opens). Thus, we are done by the above lemma. 
\end{proof}

\begin{cor}
Restriction along the identification \[\mathrm{N}(\mathcal{W}(X)) \xrightarrow{\mathrm{top}} \mathrm{N}(\left\{\mathrm{M}(W) \in \mathrm{Opens}(\mathrm{M}(X)):W \in \mathcal{W}(X)\right\})\] induces an equivalence of \(\infty\)-categories

\[\mathrm{Shv}_{\mathcal{B}}(\mathrm{M}(X),\mathcal{C})\rightarrow \mathrm{Shv}_{wo}(X,\mathcal{C}).\]

Here, \(\mathcal{B}\) denotes the basis on \(\mathrm{M}(X)\) consisting of opens of form \(\mathrm{M}(W)\) with \(W \in \mathcal{W}(X)\), and \(\mathrm{Shv}_{\mathcal{B}}\) denotes sheaves for this basis (closed under finite intersection).
\end{cor}

\begin{proof}
We can prove this by noting that the identification 

\(\mathrm{N}(\mathcal{W}(X)) \cong \mathrm{N}(\left\{\mathrm{M}(W) \in \mathrm{Opens}(\mathrm{M}(X)):W \in \mathcal{W}(X)\right\})\)

also identifies the covering sieves for either side (we will denote the right-hand-side by \(S\) for convenience). A covering sieve on some \(W \in \mathcal{W}(X)\) at left is a full subcategory of form (for an admissible wide open covering \(\mathcal{U}:=\left\{W_i\right\}_{i \in I}\) of \(W\), and denoting by \(\mathcal{W}_{\mathcal{U}}(X) \) the poset of wide opens contained in some element of \(\mathcal{U}\))
\[\mathrm{N}(\mathcal{W}_{\mathcal{U}}(X)) \times_{\mathrm{N}(\mathcal{W}(X))} \mathrm{N}(\mathcal{W}(X))_{/W}\subset \mathrm{N}(\mathcal{W}(X))_{/W}.\]

It is obvious that the equivalence \(\mathrm{top}\) carries this covering sieve to the covering sieve defined by the covering \(\mathrm{M}(\mathcal{U})\) of \(\mathrm{M}(W)\). 
\par
Conversely, given a covering of form \(\mathrm{M}(\mathcal{U})\) of some \(\mathrm{M}(W)\), we know \(\mathcal{U}\) is an admissible covering of \(W\), so the covering sieve corresponding to \(\mathrm{M}(\mathcal{U})\) is identified with the covering sieve associated to \(\mathcal{U}\) under \(\mathrm{top}\).

\end{proof}

We also have a similar result to Appendix Proposition 5.10 for \itshape finite unions \upshape of quasi-compact admissible opens.

\begin{prop}
Let \(K_1,...,K_n,K\) be quasi-compact admissible opens of \(X\). We have \(K_1 \cup \cdots \cup K_n = K\) if and only if \(\mathrm{M}(K_1)\cup \cdots \cup \mathrm{M}(K_n) = \mathrm{M}(K)\). 
\end{prop}

\begin{proof}
One direction follows easily from compatibility of \(\mathrm{M}(-)\) with admissible unions. For the other, suppose that \(\mathrm{M}(K_1) \cup \cdots \cup \mathrm{M}(K_n) = \mathrm{M}(K)\). This implies \(\mathrm{M}(K_1 \cup \cdots \cup K_n) = \mathrm{M}(K)\). However, this implies \(K_1 \cup \cdots \cup K_n = K\), as desired. 
\end{proof}

\bibliographystyle{amsplain}

\end{document}